\documentclass[11pt]{amsart}

\usepackage{amsfonts, amssymb, amsmath, amsthm, mathrsfs,euscript,color}
\usepackage{geometry}
\usepackage{graphicx}
\usepackage{epstopdf}
\usepackage{url}
\usepackage[all]{xy}

\theoremstyle{theorem}
\newtheorem*{thm*}{Theorem}
\newtheorem{thm}{Theorem}[section]
\newtheorem{prop}[thm]{Proposition}
\newtheorem*{prop*}{Proposition}
\newtheorem*{lem*}{Lemma}
\newtheorem{lem}[thm]{Lemma}
\newtheorem*{cor*}{Corollary}
\newtheorem{cor}[thm]{Corollary}

\newtheorem*{prob*}{Problem}
\newtheorem{coro}[thm]{Corollary}
\newtheorem*{coro*}{Corollary}
\newtheorem{definition}[thm]{Definition}
\newtheorem*{definition*}{Definition}
\newtheorem*{remark*}{Remark}
\newtheorem{remark}[thm]{Remark}
\newtheorem*{example*}{Example}
\newtheorem{example}[thm]{Example}

\newtheorem{ques}[thm]{Question}

\newtheorem*{thma*}{Theorem A}
\newtheorem*{thmb*}{Theorem B}
\newtheorem*{thmc*}{Theorem C}
\newtheorem*{thmd*}{Theorem D}

\newtheorem*{qa*}{Question A}
\newtheorem*{qb*}{Question B}
\newtheorem*{qc*}{Question C}

\numberwithin{equation}{section}

\newcommand*{\e}{\ensuremath{\epsilon}}

\newcommand*{\Prim}{\ensuremath{\text{\upshape P}}}
\newcommand*{\GKdim}{\ensuremath{\text{\upshape GKdim}}}

\newcommand*{\Hom}{\ensuremath{\text{\upshape Hom}}}

\newcommand*{\gr}{\ensuremath{\text{\upshape gr}}}
\newcommand*{\HL}{\ensuremath{\text{\upshape H}}}

\newcommand*{\Ker}{\ensuremath{\text{\upshape Ker}}}

\newcommand*{\ad}{\ensuremath{\text{\upshape ad}}}
\newcommand*{\End}{\ensuremath{\text{\upshape End}}}
\newcommand*{\Img}{\ensuremath{\text{\upshape Im}}}
\newcommand*{\Id}{\ensuremath{\text{\upshape Id}}}
\newcommand*{\GL}{\ensuremath{\text{\upshape GL}}}
\newcommand*{\PGL}{\ensuremath{\text{\upshape PGL}}}

\newcommand*{\Der}{\ensuremath{\text{\upshape Der}}}

\newcommand*{\Aut}{\ensuremath{\text{\upshape Aut}}}

\newcommand*{\cocy}{\ensuremath{\text{\upshape Z}}}
\newcommand*{\cobo}{\ensuremath{\text{\upshape B}}}
\newcommand*{\EL}{\ensuremath{\mathcal E}}
\newcommand*{\LL}{\ensuremath{\mathcal L}}
\newcommand*{\HG}{\ensuremath{\mathcal H}}

\newcommand*{\Opext}{\ensuremath{\text{\upshape Opext}}}

\newcommand*{\Bock}{\boldsymbol{\omega}}

\newcommand*{\ZO}{\boldsymbol{\Phi}}

\newcommand*{\Cobar}{\ensuremath{\Omega}}

\newcommand*{\diff}{\ensuremath{d}}

\newcommand*{\RT}{\ensuremath{\mathcal D}}

\newcommand*{\PM}{\ensuremath{\mathcal P}}

\newcommand*{\ID}{\ensuremath{x}}  

\newcommand*{\CH}{\ensuremath{\mathscr X}}   

\newcommand*{\CR}{\ensuremath{\mathscr T}}  

\newcommand*{\AMP}{\ensuremath{\mathscr G}}  

\newcommand*{\pfield}{\ensuremath{\mathbb K}}  

\begin{document}

\title{Isomorphism classes of finite dimensional connected Hopf algebras in positive characteristic}

\author{Xingting Wang}
\address{
Department of Mathematics\\ University of California, San Diego\\ La Jolla, CA, 92093}
\email{xiw199@ucsd.edu}



\begin{abstract}
We classify all finite-dimensional connected Hopf algebras with large abelian primitive spaces.
We show that they are Hopf algebra extensions of restricted enveloping algebras of certain restricted Lie algebras.
For any abelian matched pair associated with these extensions, we construct a cohomology group, which classifies all the extensions up to equivalence.
Moreover, we present a $1$-$1$ correspondence between the isomorphism classes and a group quotient of the cohomology group deleting some exceptional points, 
where the group respects the automorphisms of the abelian matched pair and the exceptional points represent those restricted Lie algebra extensions.
\end{abstract}

\maketitle

\tableofcontents
\setcounter{section}{-1}

\section{Introduction}
Hopf algebras originated naturally in algebraic topology relating to the concept of $H$-space. Recently, the study of Hopf algebras has become increasingly importance. For instance, quasitriangular Hopf algebras, introduced by Drinfeld \cite{Dr}, supply solutions to the quantum Yang-Baxter equations arising from quantum field theory. The representations of any Hopf algebra form a tensor category, which imparts quantum invariants for knots, links and 3-manifolds \cite{Tu}.

The classification problem plays a central role in understanding the structure of Hopf algebras. During the last two decades, great progress has been made towards this direction by imposing certain hypotheses on the Hopf algebra itself, i.e., semisimplicity, pointedness of abelian group-like elements, triangular Hopf algebras, etc. See the surveys \cite{An,Bea,BG1,BG2} on finite-dimensional Hopf algebras. Nonetheless, almost all the results are achieved over an algebraically closed field of characteristic zero. On the contrary, in positive characteristic $p$, even $p$-dimensional Hopf algebras are not yet classified, while any Hopf algebra of prime dimension over an algebraically closed field of characteristic 0 is a group algebra \cite{Zhu}.

In this paper, we will employ the following notation. Let $k$ be a base field of characteristic $p>0$ (additionally, $k$ is assumed to be perfect in Section \ref{S:Realization}, and $k$ is algebraically closed in Section \ref{S:last}), and reserve $\pfield$ for the finite field with $p$ elements. We use the standard notation $(H, m, u, \Delta, \e, S)$ as in \cite{MO93} to denote a Hopf algebra. The \emph{primitive space} of $H$ is $\Prim(H):=\{x\in H|\Delta(x)=x\otimes 1+1\otimes x\}$. In characteristic $p>0$, $\Prim(H)$ is a restricted Lie algebra, where the Lie bracket and restricted map are given by the commutator and $p$-th power in $H$. The \emph{restricted universal enveloping algebra} of $\Prim(H)$ is denoted by $u(\Prim(H))$ \cite[Chapter V, \S 7]{JA79}.
The \emph{augmentation ideal} of $H$ is denoted by $H^+:=\ker\e$. When $H$ is finite-dimensional, we denote by $H^*$ the \emph{dual Hopf algebra} of $H$. We write $H^n$ for the $n$-fold tensor product $H^{\otimes n}$ and $1$ for the identity map on $H$. We are interested in the classification of finite-dimensional connected Hopf algebras in positive characteristic. 
\begin{definition}
The \emph{coradical} $H_0$ of $H$ is the sum of all simple subcoalgebras of $H$. We say that $H$ is \emph{connected} if $H_0$ is one-dimensional. 
\end{definition}
In the literature, connected Hopf algebras are often called by different names, such as irreducible or co-connected Hopf algebras. The first well-known classification result for connected Hopf algebras is due to Milnor, Moore, Cartier and Kostant around 1963. It says that any cocommutative connected Hopf algebra over the complex numbers is a universal enveloping algebra of a Lie algebra. Recently, the non-cocommutative version has been extended up to GK-dimension 4; see \cite{BZ1, BZ2, Zh1, WZZ2}. Note that finite-dimensional (or GK-dimension 0) connected Hopf algebras only appear in positive characteristic, and they all have dimension $p^n$ for some $n\ge 0$ \cite[Proposition 1.1(1)]{sscHa}. Graded, cocommutative, connected Hopf algebras of dimension $p^2$ and $p^3$ are classified by Henderson \cite{Hen} using Singer's theory \cite{Singer} of extensions of connected Hopf algebras. The non-graded, non-cocommutative version has been studied lately in \cite{Wan, LinXT2} using the theories of restricted Lie algebras and Hochschild cohomology of coalgebras. So far, a complete classification has been obtained by explicit generators and relations for connected Hopf algebras of dimension $\le p^3$ except for the following case in dimension $p^3$.
\begin{definition}
Let $d$ be a positive integer and let $\CH(d)$ be the set consisting of all finite dimensional connected Hopf algebras $H$ satisfying the following two conditions:
\begin{itemize}
\item[(a)] $\dim H=p^{d+1}$,
\item[(b)] $\Prim(H)$ is an abelian restricted Lie algebra of dimension $d$.
\end{itemize}
\end{definition}
One predicament of classifying $\CH:=\bigcup \CH(d)$ is that it contains numerous parametric families; see for example the following parametric family in $\CH(2)$:
\begin{example}\label{C0:Exp}
Let $A(\lambda)$ be the quotient algebra of the free algebra $k\langle x,y,z\rangle$, subject to the relations:
\begin{align*}
x^p=0,\ y^p=y,\ z^p+x^{p-1}y=x,\ [x,y]=0,\ [z,x]=y,[z,y]=0.
\end{align*}
Regarding the coalgebra structure, $x$ and $y$ are primitive and the comultiplication of $z$ is given by
\begin{align*}
\Delta(z)=z\otimes 1+1\otimes z+\lambda x\otimes y+\sum_{1\le i\le p-1}{p\choose i}/p\ x^i\otimes x^{p-i}.
\end{align*}
The isomorphism classes in $A(\lambda)$ regarding the parameter $\lambda$ are more subtle. Suppose $k$ is algebraically closed. Let $G$ be the multiplicative group $\sqrt[\frac{p^2-1}{2}]{1}$ acting on $k$ by scalar multiplication. Indeed $A(\lambda)\cong A(\lambda')$ if and only if $\lambda,\lambda'$ are in the same $G$-orbits of $k$. 
\end{example}

The purpose of this paper is to describe all Hopf algebras in $\mathscr X$ and their isomorphism classes explicitly. We will show that every Hopf algebra $H$ in $\mathscr X$ fits into a \emph{cocentral} extension
\begin{align}\label{E:SES}
\xymatrix{
1\ar[r]&
A\ar[r]^-{\iota}&
H\ar[r]^-{\pi}&
B\ar[r]&
1
}
\end{align}
of finite-dimensional Hopf algebras, where $A, B$ are the restricted universal enveloping algebras of some abelian restricted Lie algebras, i.e., $A=u(\mathfrak h)$ with $\mathfrak h$ abelian, and $B=u(\mathfrak g)$ with $\dim \mathfrak g=1$. By saying that the extension is cocentral, we mean that the dual Hopf algebra $B^*$ is central in $H^*$. Let us generalize $B$ so that $B$ is generated by one non-zero primitive, say $z$. Then $B$ is necessarily of the form $B=k[z]/f(z)$, where
\[f(z):=z^{p^n}+\lambda_{n-1}z^{p^{n-1}}+\cdots +\lambda_1z^p+\lambda_0z=0,\ \text{for}\ n\ge 1.\]
The primitive space $\Prim(B)$ of $B$ is denoted by $\mathfrak g$. Note that $n=1$ is the previous special situation, in which case we write $f(z)=z^p+\lambda z$ for some $\lambda \in k$. We construct a Hopf algebra $u(\mathscr D)$ from the explicit data 
\begin{align}\label{Data}
\mathscr{D}=\big(T, z, \Theta, \chi\big),
\end{align}
where $T$ gives an action of $B$ on $A$, the element $\Theta$ is in the augmentation ideal $A^+$ of $A$ and $\chi$ is some $2$-cocycle in $A^+\otimes A^+$ regarding the cobar construction on $A$; see details in Section \ref{S:AGC}. Our first theorem, precisely stated below, proves that the Hopf algebra $u(\mathscr D)$ indeed fits into a cocentral extension as above (and every such extension is given by some $u(\mathscr D)$). Moreover, the theorem gives a necessary and sufficient condition for $\Prim(u(\mathscr D))$ to coincide with $\mathfrak h=\Prim(A)$; this plays an important role in our classification result.
\begin{thm}\label{C0:thmA}
The following hold for $u(\mathscr D)$:
\begin{itemize}
\item[(a)] $u(\mathscr D)$ is a connected Hopf algebra of dimension $p^{\dim \Prim(A)+n}$.
\item[(b)] $u(\mathscr D)$ is an extension of $A$ by $B$.
\item[(c)] the primitive space of $u(\mathscr D)$ is isomorphic to $\mathfrak h$ if and only if $\{[\RT_z^i(\chi)]|0\le i\le n-1\}$ are linearly independent in $\HL^2(\Cobar A)$. (See Definitions \ref{C1:Cobar} and \ref{C2:Def1}.)
\end{itemize}
\end{thm}
Theorem \ref{C0:thmA} will be proved in Section \ref{S:thmA}. Moreover, every Hopf algebra in $\CH$ is a $B$-extension over $A$ satisfying $\dim B=p$, hence is given by some data $\mathscr D$ (Corollary \ref{C4:ObjH}). Next, since $A$ is commutative and $B$ is cocommutative, every extension of $A$ by $B$ is naturally associated with an abelian matched pair $(\rightharpoonup, \varrho)$, which consists of an action and a coaction 
\[
\rightharpoonup: B\otimes A\to A,\  \varrho: B\to B\otimes A
\]
satisfying certain conditions \cite[Definition 2.2]{MK2}. For abelian matched pairs associated with $\CH$, the coaction is always trivial and the action is induced by some algebraic representation $\rho$ of $\mathfrak g$ on $\mathfrak h$ (Lemma \ref{C4:Apair}). Hence they can be described by $T=(\mathfrak g,\mathfrak h, \rho)$, where $T$ is called a \emph{type}. Now fix a type $T=(\mathfrak g,\mathfrak h,\rho)$. We form a cohomology group $\HG^2(B,A)$ consisting equivalence classes of $(\chi,\Theta)$ in the data $\mathscr D$. With the help of Theorem \ref{C0:thmA} (c), we further define a subgroup $\HG^2(\mathfrak g,\mathfrak h)$ representing primitively generated extensions. Then we have our second result.

\begin{thm}\label{C0:thmB}
There exist the following $1$-$1$ correspondences:
\begin{itemize}
\item[(a)] View $\mathfrak h$ as a (left) restricted $\mathfrak g$-module via $\rho$. 
\[\xymatrix{
\left\{\mbox{equiv-classes of restricted Lie algebra extensions of}\  \mathfrak h\ \mbox{by}\  \mathfrak g\right\}\ar@{<->}[r]  & \left\{\mbox{elements of}\ \HG^2(\mathfrak g,\mathfrak h)\right\}.}   
\]
\item[(b)] Suppose the abelian matched pair is given by $T$. 
\[\xymatrix{
\left\{\mbox{equiv-classes of Hopf algebra extensions of type}\ T\right\}\ar@{<->}[r] & \left\{\mbox{elements of}\ \HG^2(B,A)\right\}.}\]
\end{itemize}
\end{thm}
As a matter of fact, the complement $\HG^2(B,A)\setminus\HG^2(\mathfrak g,\mathfrak h):=\HG^2(T)$ relates equivalent classes of $\CH$. We denote by $\Aut(T)$ the automorphism group of $T$ by considering the category of all types as a full subcategory of the category defined in \cite[\S 1]{Ho} for all abelian matched paris (Corollary \ref{C4:subcat}). Moreover, there is a natural group action of $\Aut(T)$ on $\HG^2(T)$ induced by the Hopf algebra isomorphisms described in Lemma \ref{C4:IsomMap}. Then we have our classification result.

\begin{thm}\label{C0:thmC}
For a fixed type $T$, there is a $1$-$1$ correspondence between isomorphism classes in $\CH$ and $\Aut(T)$-orbits in $\HG^2(T)$. Moreover, the isomorphism classes of $\CH$ are bijective to the disjoint union
\[
\coprod_{T}\HG^2(T)/\Aut(T),
\]
where $T$ runs through all non-isomorphic types.
\end{thm}
Theorem \ref {C0:thmB} and Theorem \ref{C0:thmC} will be proved in Section \ref{S:Grp}. For the sake of completeness, it is verified that our cohomology group $\HG^2(B,A)$ is isomorphic to the one defined in \cite{Ho, Ma1} for computing Hopf algebra extensions associated with abelian matched pairs (Proposition \ref{C5:IsomGrp}). 

Finally, we study the $\Aut(T)$-quotient of $\HG^2(T)$ inspired by finite group quotients of affine or projective spaces. Identify $\mathfrak h=\mathbb A^d$ and $\HL^2(\Cobar A)=\mathbb A^{d(d+1)/2}$ for some $d>0$. We embed $\HG^2(T)$ into a subquotient of $\mathbb A^d\times \mathbb A^{d(d+1)/2}$ (Proposition \ref{P:Embed}). Moreover, there is a projection from $\HG^2(T)$ to $\mathbb A^{d(d+1)/2}$, whose images are called admissible elements (Definition \ref{D:Ad}). We expect that admissible elements will play a central role in understanding the structure of $\HG^2(T)$. Note that all admissible elements are $z$-characteristic (Definition \ref{D:Zco}). Under certain circumstances, they do coincide (Proposition \ref{P:TAd}), and there exists a bijection between $\HG^2(T)$ and nonzero $z$-characteristic elements in $\mathbb A^{d(d+1)/2}$ (Lemma \ref{L:Fiber}). Therefore, we develop formulas to calculate $z$-characteristic elements (Proposition \ref{C7:zchar}) as well as $\Aut(T)$-actions in $\mathbb A^{d(d+1)/2}$. As an application, all the semisimple Hopf algebras in $\CH$ are classified assuming $k$ is algebraically closed. Let $\Lambda(V)$ be the exterior algebra over any vector space $V$. A quadratic curve in $\Lambda(V)$ is defined to be a one-dimensional subspace in $V\oplus (V\wedge V)\subseteq \Lambda(V)$. We prove the following result in our last section.

\begin{thm}\label{C0:thmD}
The following isomorphism classes are in $1$-$1$ correspondence with each other.
\begin{itemize}
\item[(a)] The isomorphism classes of semisimple connected Hopf algebras of dimension $p^{d+1}$ with $\dim \Prim(H)=d$.
\item[(b)] The isomorphism classes of quadratic curves in $\mathbb P_\pfield^{d-1}$ ($p=2$) or $\Lambda(V)$ with $\dim V=d$ ($p>2$) up to the affine automorphism group $\PGL(d,\pfield)$. 
\item[(c)] The isomorphism classes of $p$-groups of order $p^{d+1}$, whose Frattini group is isomorphic to $C_p$. (The Frattini group of a $p$-group $G$ is the smallest normal subgroup $N$ such that $G/N$ is an elementary abelian $p$-group.)
\end{itemize}
\end{thm}

The paper is organized as follows. Section \ref{S:Pre} discusses $H$-module Hopf algebras, algebraic representations and their relations with cobar constructions. Section \ref{S:AGC} contains the definition of a compatible data $\mathscr D$ and the construction of $u(\mathscr D)$. Section \ref{S:thmA} provides the proof of Theorem \ref{C0:thmA}. Section \ref{S:Ext} is devoted to Hopf algebra extensions of $A$ by $B$. Section \ref{S:Grp} introduces cohomology groups and automorphism group actions to classify $\CH$. Section \ref{S:Realization} paves a way to compute the automorphism group actions on cohomology groups. Section \ref{S:last} gives the proof of Theorem \ref{C0:thmD}.

In an upcoming paper \cite{LinXT3}, we will use the results of this paper to classify $\CH(2)$, which contains 8 parametric families including $A(\lambda)$ in Example \ref{C0:Exp}. Therefore, the classification of connected Hopf algebras of dimension $\le p^3$ will be complete. We note that D.-G. Wang, J.J. Zhang and G. Zhuang in \cite{WZZ2} studied connected Hopf algebras $H$ over an algebraically closed field of characteristic 0 satisfying $\GKdim H=\dim \Prim(H)+1<\infty$. They show that such Hopf algebra is always isomorphic, as an algebra, to some universal enveloping algebra of a Lie algebra; see \cite[Theroem 0.5 or Theorem 2.7]{WZZ2}. In positive characteristic, the restricted version does not hold for finite-dimensional connected Hopf algebras generated by primitives plus a nonprimitive element. In \cite[\S 6]{LinXT2}, it is proved that the (B2) Hopf algebra in \cite[Theorem 1.3]{LinXT2} is not isomorphic, as an algebra, to any restricted universal enveloping algebras. Let $H$ be any Hopf algebra in $\CH$. Since it is generated by primitives plus a nonprimitive element, we can ask the same question.
\begin{ques}
When does $H$ have the same algebra structure as $u(\mathfrak l)$ for some restricted Lie algebra $\mathfrak l$?
\end{ques}
The question also helps us to understand the representations and Hochschild cohomology ring of $H$. More questions are presented in Section \ref{S:Grp} and Section \ref{S:Realization}.
\newline

\noindent
\textbf{Acknowledgments}
The author would like to greatly acknowledge the help and guidance from his Ph.D. advisor, Professor J.J. Zhang. The author is also thankful to Cris Negron, Linhong Wang and Guangbin Zhuang for their remarks on the earlier drafts of the paper and to Professor Brown for pointing out a typo in Reference \cite{BZ2}. The author thanks the referee for valuable suggestions and comments to largely improve the paper.

\section{Preliminary results}\label{S:Pre}
We gather some basic facts that will be used later in this paper.
\begin{definition}\label{D:HAModule}  
Let $H$ denote a Hopf algebra. A Hopf algebra $A$ is said to be an $H$-module Hopf algebra if $A$ is a (left) $H$-module satisfying that 
\begin{itemize}                                                                                                                                   
\item[(a)] $h\cdot(ab)=\sum (h_1\cdot a)(h_2\cdot b)$,
\item[(b)] $h\cdot1_A=\e(h)1_A$,
\item[(c)] $\Delta(h\cdot a)=\sum (h_1\cdot a_1)\otimes (h_2\cdot a_2)$, 
\item[(d)] $\e(h\cdot a)=\e(h)\e(a)$,
\end{itemize}
for all $h\in H$ and $a,b\in A$. 
\end{definition}
Note that $A$ is an $H$-\emph{module algebra} \cite[Definition 4.1.1]{MO93} if it satisfies (a) and (b). For an $H$-module Hopf algebra $A$, the two linear maps $h\otimes a\mapsto S(h\cdot a)$ and $h\otimes a\mapsto h\cdot S(a)$ are both convolution-inverse to the action $H\otimes A\to A$. Hence we have $h\cdot S(a)=S(h\cdot a)$ for all $h\in H$ and $a\in A$.
\begin{definition}\label{C1:Cobar}
For any Hopf algebra $A$, the cobar construction on $A$ is the differential graded algebra $\Cobar A$ defined as follows:
\begin{itemize}
\item[(a)] As a graded algebra, $\Cobar A$ is the tensor algebra $T\left(A^+\right)$,
\item[(b)] The differentials are given by
\begin{align}
\diff^n=\sum_{i=0}^{n-1}(-1)^{i+1}\ 1^i\otimes\overline{\Delta}\otimes 1^{n-i-1},\ \mbox{where}\ \overline{\Delta}(a)=\Delta(a)-a\otimes 1-1\otimes a\ \mbox{for any}\ a\in A^+. \label{D:diff}
\end{align}
\end{itemize}
\end{definition}
By definition, we have $\diff^1(a)=-\overline{\Delta}(a)$ and $\diff^2(a\otimes b)=-\overline{\Delta}(a)\otimes b+a\otimes \overline{\Delta}(b)$ for any $a,b\in A^+$; see \cite[\S 19]{Rational} for basic properties of bar and cobar constructions.
\begin{definition}\label{D:AR}
Let $\mathfrak h$ and $\mathfrak g$ be two restricted Lie algebras. An algebraic representation of $\mathfrak g$ on $\mathfrak h$ is a linear map $\rho: \mathfrak g\to \End_k(\mathfrak h)$ such that
\begin{itemize}\label{AR}
\item[(a)] $\rho_{[x,y]}=\rho_x\rho_y-\rho_y\rho_x$,
\item[(b)] $\rho_{\left(x^p\right)}=\left(\rho_x\right)^p$,
\item[(c)] $\rho_x([a,b])=[\rho_x(a),b]+[a,\rho_x(b)]$,
\item[(d)] $\rho_x(a^p)=\rho_x(a)\left(\ad\ a\right)^{p-1}$,
\end{itemize}
for any $x,y\in \mathfrak g$ and $a,b\in \mathfrak h$. 
\end{definition}
Note that $\mathfrak h$ becomes a restricted $\mathfrak g$-module via $\rho$ by (a) and (b), and there is an extension of restricted Lie algebras $0\to \mathfrak h\to \mathfrak h\rtimes \mathfrak g\to \mathfrak g\to 0$ \cite[7.4.9]{Wei}. 
In characteristic $p>0$, all the derivations on $u(\mathfrak h)$ form a restricted Lie algebra.
For any $\delta\in \Der(u(\mathfrak h))$ and $a,b\in u(\mathfrak h)$, direct computation shows that
\begin{align*}
\delta[a,b]=[\delta(a),b]+[a,\delta(b)],\quad \delta(a^p)=\delta(a)(\ad\ a)^{p-1}.
\end{align*}
Therefore by Definition \ref{AR},  any algebraic representation $\rho$ of $\mathfrak g$ on $\mathfrak h$ induces a restricted Lie algebra map from $\mathfrak g$ to $\Der(u(\mathfrak h))$. It makes $u(\mathfrak h)$ into a $u(\mathfrak g)$-module Hopf algebra. Moreover, we have

\begin{prop}\label{eqdef}
Algebraic representations of $\mathfrak g$ on $\mathfrak h$ are in $1$-$1$ correspondence with $u(\mathfrak g)$-module Hopf algebra structures on $u(\mathfrak h)$.
\end{prop}
\begin{proof}
By above argument, it suffices to show the other direction. Suppose $u(\mathfrak h)$ is a $u(\mathfrak g)$-module Hopf algebra via the action $\rightharpoonup$. We define the representation $\rho$ by $\rho_x(a)=x\rightharpoonup a$ for any $x\in \mathfrak g$ and $a\in \mathfrak h$. By Definition \ref{D:HAModule}, it is easy to show $\rho: \mathfrak g\to \End_k(\mathfrak h)$ is an algebraic representation. The bijection comes from the explicit constructions. 
\end{proof}

Let $H$ denote a Hopf algebra, and $A$ be an $H$-module Hopf algebra. Note that $A^+$ is invariant under the $H$-action by Definition \ref{D:HAModule} (d). 
Hence, we can consider $\Cobar A=T(A^+)$ as an $H$-module algebra via the comultiplication of $H$. 
In details, 
\begin{align*}
h\cdot 1_{\Cobar A}=&\ \e(h)1_{\Cobar A},\\
h\cdot \left(a_1\otimes a_2\otimes \cdots \otimes a_n\right)=&\ \sum (h_1\cdot a_1)\otimes(h_2\cdot a_2)\otimes \cdots \otimes (h_n\cdot a_n),
\end{align*}
for any $h\in H$ and $a_i\in A^+$. Moreover, we can pass the $H$-module algebra structure onto the cohomology ring of $\Cobar A$.
\begin{prop}\label{CRHModule}
The $H$-action commutes with the differentials of $\Cobar A$. In particular, the cohomology ring $\HL^\bullet(\Cobar A)$ becomes an $H$-module algebra.
\end{prop}
\begin{proof}
The cobar construction $\Cobar A$ is a differential graded algebra generated in degree one. Thus, it suffices to show that the $H$-action commutes with $\diff^1$, which is easy to check.
\end{proof}
As a consequence, we have
\begin{coro}\label{S:ARCR}
Let $\rho$ be an algebraic representation of $\mathfrak g$ on $\mathfrak h$. Then $\HL^\bullet(\Cobar u(\mathfrak h))$ becomes a $u(\mathfrak g)$-module algebra via the representation $\rho$. 
\end{coro}

\section{A general construction}\label{S:AGC}
We will first explain the data $\mathscr D$ in the introduction. We keep the same notations, where $A=u(\mathfrak h)$ with $\mathfrak h$ abelian and $B=k[z]/f(z)$ with primitive generator $z$ satisfying 
\[
f(z)=z^{p^n}+\lambda_{n-1}z^{p^{n-1}}+\cdots+\lambda_1z^p+\lambda_0z=0,\ \mbox{for}\ n\ge 1.
\]
We use $\mathfrak g$ to denote $\Prim(B)$. Let $\rho$ be an algebraic representation of $\mathfrak g$ on $\mathfrak h$. We denote $T=(\mathfrak g,\mathfrak h,\rho)$, and $T$ is said to be a \emph{type}. We use $\Hom_{\gr \pfield}(\Cobar A,\Cobar A)$ to denote all the $\pfield$-linear graded maps from $\Cobar A$ to itself. Note that by Corollary \ref{S:ARCR}, $\Cobar A$ is a $B$-module algebra via $\rho$. Since the $B$-action commutes with the differentials in $\Cobar A$, we can consider $\rho_z$ as a degree zero cochain map from $\Cobar A$ to itself.
\begin{definition}\label{C2:Def1}
We define three degree zero cochain maps in $\Hom_{\gr \pfield}(\Cobar A,\Cobar A)$. 
\begin{itemize}
\item[(a)] For any element $a\in\left(A^+\right)^n$, the $p$-th power map $\PM$ is given by $\PM(a)=a^p$.
\item[(b)] Inductively, we define $\RT_z^0=\Id$ and $\RT_z^m=\PM\circ\RT_z^{m-1}+\rho_z^{p^m-p^{m-1}}\circ\RT_z^{m-1}$ for any $m\ge 1$.
\item[(c)] The $z$-operator on $\Cobar A$ is given by $\ZO_z:=\RT_z^{n}+\lambda_{n-1}\RT_z^{n-1}+\cdots+\lambda_1\RT_z^1+\lambda_0\RT_z^0$.
\end{itemize}
\end{definition}
Throughout the paper, by abuse of notations, we will also consider $\PM,\RT_z^m,\ZO_z$ as maps from $\left(A^+\right)^m$ to itself for any integer $m\ge 0$. All these maps appear naturally in the following lemma. 

\begin{lem}\label{AppOperator}
Let $F$ be the algebra generated by $A$ and an indeterminate $\ID$, subject to the relations: $[\ID, a]=\rho_z(a)$ for all $a\in A$. Suppose $f\in A^+\otimes A^+$. In the algebra of tensor product $F\otimes F$, we have
\[
\left(\ID\otimes 1+1\otimes \ID+f\right)^{p^m}=\ \ID^{p^m}\otimes 1+1\otimes \ID^{p^m}+\RT_z^{m}(f),
\]
for all $m\ge 0$.
\end{lem}
\begin{proof}
We will prove the statement by induction on $m\ge 0$. 
The statement is trivial for $m=0$. 
Suppose it is true for $m=n$. Write $X=\ID \otimes 1+1\otimes \ID$ and observe that
\[
(\ad X^{p^n})^{p-1}=\left(\ad X\right)^{p^{n+1}-p^n}=\left(\ad \ID \otimes 1+1\otimes \ad \ID\right)^{p^{n+1}-p^n}=\rho_z^{p^{n+1}-p^n}.
\]
Hence by \cite[Lemma A.1]{Wan} ($A$ is commutative), we have
\begin{align*}
(X+f)^{p^{n+1}}=&\ [X^{p^n}+\RT_z^{n}(f)]^p\\
=&\ X^{p^{n+1}}+(\RT_z^{n}(f))^p+(\ad X^{p^n})^{p-1}\RT_z^n(f)\\
=&\ X^{p^{n+1}}+(\PM\circ \RT_z^{n}+\rho_z^{p^{n+1}-p^n}\circ \RT_z^n)(f)\\
=&\ \ID^{p^{n+1}}\otimes 1+1\otimes \ID^{p^{n+1}}+\RT_z^{n+1}(f).
\end{align*}
\end{proof}
We list some basic properties regarding these maps we have defined.
\begin{prop}\label{Maps}
The following are true:
\begin{itemize}
\item[(a)] All maps $\PM,\RT_z^m$ and $\ZO_z$ commute with the differentials of $\Cobar A$.
\item[(b)] $\RT_z^m=\PM\circ \RT_z^{m-1}+\rho_z^{p^m-1}$ for all $m\ge 1$.
\item[(c)] $\rho_z\circ \ZO_z=0$.
\end{itemize}
\end{prop}
\begin{proof}
(a) Easy. (b) Induction using the fact $\rho_z\circ\PM=0$.

(c) By definition and (b), we have
\begin{align*}
\rho_z\circ \ZO_z=&\ \rho_z\circ (\RT_z^{n}+\cdots+\lambda_1\RT_z^1+\lambda_0\RT_z^0)\\
=&\ \rho_z\circ \left(\PM\circ\RT_z^{n-1}+\cdots+\lambda_1\PM\circ \RT_z^{0}\right)+\rho_z(\rho_z^{p^n-1}+\cdots+\lambda_1\ \rho_z^{p-1}+\lambda_0)\\
=&\ \left(\rho_z\right)^{p^n}+\lambda_{n-1} \left(\rho_z\right)^{p^{n-1}}+\cdots+\lambda_1 \left(\rho_z\right)^{p}+\lambda_0 \rho_z\\
=&\ \rho_{\left(z^{p^n}+\lambda_{n-1}z^{p^{n-1}}+\cdots+\lambda_1 z^p+\lambda_0z\right)}\\
=&\ 0.
\end{align*}
\end{proof}
By Proposition \ref{Maps} (a), we can view $\ZO_z$ as a $\pfield$-linear map from $\HL^\bullet(\Cobar A)$ to itself of degree zero. Now, we are able to describe the element $\chi$ in \eqref{Data}.
\begin{definition}\label{D:Zco}
Let $\chi\in A^+\otimes A^+$. We say that $\chi$ is a $z$-cocycle if 
\begin{itemize}
\item[(a)] $\chi\in \cocy^2(\Cobar A)$, 
\item[(b)] $\ZO_z(\chi)\in \cobo^2(\Cobar A)$. 
\end{itemize}
Moreover, any cohomology class $\xi\in \HL^2(\Cobar A)$ is said to be $z$-characteristic if $\ZO_z(\xi)=0$. 
\end{definition}
Note that $\chi$ is a $z$-cocycle if and only if $[\chi]$ is $z$-characteristic in $\HL^2(\Cobar A)$. Finally, we give the construction of $u(\mathscr D)$ by using the date $\mathscr D$, where $\Theta\in A^+$ and $\chi\in (A^+)^2$ is a $z$-cocycle satisfying 
\begin{align}\label{E:CCon}
\rho_z(\Theta)=0,\ \ZO_z(\chi)=\diff^1(\Theta).
\end{align} 

Explicitly, $u(\mathscr D)$ is generated by $A$ and an indeterminate $\ID$, subject to the relations
\begin{gather}\label{RHD}
\ID^{p^n}+\lambda_{n-1} \ID^{p^{n-1}}+\cdots+\lambda_1\ID^p+\lambda_0 \ID+\Theta=0,\ [\ID,a]=\rho_z(a),
\end{gather}
for all $a\in A$. The comultiplication of $\ID$ is given by
\begin{align}\label{CHD}
\Delta(\ID)=\ID\otimes 1+1\otimes \ID+\chi.
\end{align}

In the following, we show there is a PBW basis for $u(\mathscr D)$, which covers the dimensionality in Theorem \ref{C0:thmA} (a). 
 
\begin{lem}\label{basis}
As a left $A$-module, $u(\mathscr D)$ is free with basis $\{x^i|0\le i\le p^n-1\}$. 
\end{lem}
\begin{proof}
We apply the left version of Bergman \cite[Proposition 7.1]{Beg}, in which we should replace the base ring by $A$. It is clear that the following are all the possible ambiguities:
\begin{align}
(xa)a'=&x(aa'),\tag{a}\\
\ID(\ID^{p^n})=& \left(\ID^2\right)\ID^{p^n-1},\tag{b}\\
(\ID^{p^n})a=& \ID^{p^n-1}(\ID a),\tag{c}
\end{align}
for all $a,a'\in A$. It is easy to check that (a) and (b) are resolvable. We will do the reduction for (c).
By induction, we have $\ID^{s-1}\left(\ID a\right)=\sum_{j=0}^{s}{s\choose j}\ \rho_z^{j}(a)\ \ID^{s-j}$, for all $s\ge 1$. Hence,
\begin{align*}
\ID^{p^n-1}(\ID a)=a(\ID^{p^n})+\rho_z^{p^n}(a)=a(-\sum_{i=0}^{n-1}\lambda_i\ \ID^{p^i}-\Theta)+\rho_z^{p^n}(a).
\end{align*}
On the other hand,
\begin{align*}
(\ID^{p^n})a=&\ (-\sum_{i=0}^{n-1}\lambda_i\ \ID^{p^i}-\Theta)a=a(-\sum_{i=0}^{n-1}\lambda_i\ \ID^{p^i}-\Theta)-\sum_{i=0}^{n-1}\lambda_i\ \rho_z^{p^i}(a)\\
=&\ a(-\sum_{i=0}^{n-1}\lambda_i\ \ID^{p^i}-\Theta)+\rho_{\left(-\sum_{i=0}^{n-1}\lambda_iz^{p^i}\right)}(a)=a(-\sum_{i=0}^{n-1}\lambda_i\ \ID^{p^i}-\Theta)+\rho_z^{p^n}\left(a\right).
\end{align*}
Therefore, (c) is resolvable.
\end{proof}

\section{Proof of Theorem \ref{C0:thmA}}\label{S:thmA}
\begin{lem}\label{pinA}
Let $a\in u(\mathscr D)$. Then we have $\Delta(a)-a\otimes 1-1\otimes a\in A\otimes A$ if and only if $a\in A+\sum_{i=0}^{n-1}k\ID^{p^i}$.
\end{lem}
\begin{proof}
By Lemma \ref{basis}, we can write every element $a\in u(\mathscr D)$ as
\begin{align}\label{P:Formula1}
a=\sum_{i=0}^{p^n-1}a_i\ID^{i},
\end{align}
for some $a_i\in A$ in a unique way. One direction of the proof is clear. The other direction follows exactly the same argument in the proof of \cite[Theorem 4.5]{Wan}.
\end{proof}
\noindent
\textbf{Proof of Theorem \ref{C0:thmA} (a).} By Lemma \ref{basis}, it suffices to show that $u(\mathscr D)$ is a bialgebra. In second part (b), we prove that $u(\mathscr D)$ is an extension between two connected ones, which is necessarily connected, as is seen from the proof of \cite[Lemma 4]{TK79}. Moreover, the antipode exists automatically \cite[Lemma 14]{MO93}. 

Denote by $F$ the algebra generated by $A$ and $\ID$, subject to the relations $[\ID,a]=\rho_z(a)$, for all $a\in A$. 
The comultiplication of $x$ is given by \eqref{CHD}. It is easy to check $\Delta$ extends to an algebra map from $F$ to $F\otimes F$.  
%
%
After setting $\e(\ID)=0$, $(F,m,u,\Delta,\e)$ becomes a well-defined bialgebra. 

Next, we denote $X=\ID^{p^n}+\lambda_{n-1} \ID^{p^{n-1}}+\cdots+\lambda_0 \ID+\Theta$ as an element in F. Because $u(\mathscr D)\simeq F/(X)$, it suffices to show that $(X)$ is a bi-ideal in $F$. According to Lemma \ref{AppOperator}, we have
\[
\Delta(\ID^{p^s})=(\ID \otimes 1+1\otimes \ID+\chi)^{p^s}=\ID^{p^s}\otimes 1+1\otimes \ID^{p^s}+\RT_z^s(\chi),
\]
for all $s\ge 0$. Thus
\begin{align*}
\Delta(X)=&\ X\otimes 1+1\otimes X+\left(\RT_z^n(\chi)+\lambda_{n-1}\RT_z^{n-1}(\chi)+\cdots+\lambda_0\chi\right)+\left(\Delta(\Theta)-\Theta\otimes 1-1\otimes \Theta\right)\\
=&\ X\otimes 1+1\otimes X+\ZO_z(\chi)-d^1(\Theta)\\
=&\ X\otimes 1+1\otimes X.
\end{align*}
One sees that $\e(X)=0$ since $\Theta\in A^+$, which gives the conclusion. 
\newline

\noindent
\textbf{Proof of Theorem \ref{C0:thmA} (b).}
It is clear that $A$ is a normal subalgebra, i.e., $A^+u(\mathscr D)=u(\mathscr D)A^+$ according to \eqref{RHD}. Moreover, the quotient algebra $u(\mathscr D)/A^+u(\mathscr D)$ is isomorphic to $B$ by the PBW basis. Then the result follows from \cite[Lemma-Definition 1.1]{NsubHA}.
\newline

\noindent
\textbf{Proof of Theorem \ref{C0:thmA} (c).} Let $a\in \Prim\left(u(\mathscr D)\right)$. By previous lemma, we can write
\begin{align*}
a=b+\sum_{i=0}^{n-1}\mu_i\ID^{p^i},
\end{align*}
for some $b\in A$ and $\mu_i\in k$. Note that 
\[
d^1(x^{p^i})=x^{p^i}\otimes 1+1\otimes x^{p^i}-\Delta(x^{p^i})=x^{p^i}\otimes 1+1\otimes x^{p^i}-(x\otimes 1+1\otimes x+\chi)^{p^i}=-\RT_z^i(\chi)
\]
by Lemma \ref{AppOperator}. Since $d^1(a)=0$, we have
\begin{align*}
d^1\left(b\right)=-\sum_{i=0}^{n-1}\mu_i d^1\left(\ID^{p^i}\right)=\sum_{i=0}^{n-1}\mu_i\RT_z^i(\chi).
\end{align*}
By passing to the cohomology $\HL^2(\Cobar A)$, we have $\sum_{i=0}^{n-1}\mu_i[\RT_z^i(\chi)]=0$. For one direction, suppose $\{[\RT_z^i(\chi)]|0\le i\le n-1\}$ are linearly independent in $\HL^2(\Cobar A)$. 
Then all $\mu_i=0$ and $a=b\in A$, which implies $\Prim\left(u(\mathscr D)\right)=\Prim(A)\simeq \mathfrak h$.
On the other hand, there are coefficients $\mu_i\in k$, not all zero, such that $\sum_{i=0}^{n-1}\mu_i\RT_z^i(\chi)=d^1(c)$ for some element $c\in A^+$. Direct computation shows that $\sum_{i=0}^{n-1}\mu_i\ID^{p^i}+c$ is primitive, which is certainly not in $A$. 

\section{Extensions of connected Hopf algebras}\label{S:Ext}
In this section, we consider Hopf algebra extensions described by \eqref{E:SES} in the introduction. We keep the same notations, where $A=u(\mathfrak h)$ with $\mathfrak h$ abelian and $B=k[z]/f(z)$ with primitive generator $z$ satisfying 
\[
f(z)=z^{p^n}+\lambda_{n-1}z^{p^{n-1}}+\cdots+\lambda_1z^p+\lambda_0z=0,\ \mbox{for}\ n\ge 1.
\]
We use $\mathfrak g$ to denote $\Prim(B)$. In particular if $\dim B=p$, we will prove that any Hopf algebra extension of $A$ by $B$ is given by some $\mathscr D$. First of all, we study the algebra structure of the extension. In fact, every Hopf algebra extension of $A$ by $B$ is \emph{cleft} \cite{normalbasis}, which means that there exists a convolution-invertible comodule map $\phi: B\to H$. We can assume $\phi$ preserves the units and counits and such $\phi$ is called a \emph{section}. Then as algebras, $H$ is isomorphic to the crossed product  $A\#_\sigma B$ \cite[Theorem 11]{DT86}, where the action $\rightharpoonup$ and the $2$-cocycle $\sigma$ are determined by
\begin{align*}
b\rightharpoonup a&=\sum \phi(b_1)a\phi^{-1}(b_2),\\
\sigma(b,c)&=\sum \phi(b_1)\phi(c_1)\phi^{-1}(b_2c_2)
\end{align*}
for all $b,c\in B$ and $a\in A$. Note that $\sigma$ satisfies the cocycle condition described in \cite[Lemma 10]{DT86}.

\begin{lem}\label{C3:cleft}
Let $H$ be a $B$-cleft extension over $A$, and $\phi: B\to H$ a section. Then we have
\begin{itemize}
\item[(a)] $z\rightharpoonup a=[\phi(z),a]$,
\item[(b)] $z\rightharpoonup f(\phi(z))=0$,
\item[(c)] $f(\phi(z))\in A^+$
\end{itemize}
for all $a\in A$ and $f(\phi(z))=\phi(z)^{p^n}+\cdots +\lambda_1\phi(z)^p+\lambda_0\phi(z)$.
\end{lem}
\begin{proof}
(a) and (b) Easy. 

(c) It suffices to show $B$ coacts trivially on $f(\phi(z))$. Since $\phi$ is a comodule map,
\begin{align*}
(1\otimes \pi)\Delta[f(\phi(z))]=&f\left((1\otimes \pi)\Delta\phi(z)\right)=f\left((\phi\otimes 1)\Delta(z)\right)=f(\phi(z)\otimes 1+1\otimes z)\\
=&f(\phi(z))\otimes 1+1\otimes f(z)=f(\phi(z))\otimes 1.
\end{align*}
\end{proof}

The pair $(\delta,\Theta)\in \Der_k(A)\times A^+$ satisfying $\delta(\Theta)=0$ is called \emph{cleft data} for $B$ over $A$. The associated cleft extension is constructed as follows: it is generated by $A$ and $x$, subject to the relations:
\[
xa=ax+\delta(a),\ f(x)+\Theta=0,\ \text{for all}\ a\in A. 
\]
\begin{prop}\label{C3:cleftD}
Every $B$-cleft extension over $A$ is given by some cleft data $(\delta,\Theta)$. Moreover if the section $\phi$ satisfies $\phi(z^i)=\left(\phi(z)\right)^i$ for $0\le i\le p^n-1$, then $\sigma(z^{p^n-1},z)=\sigma(z,z^{p^n-1})=f\left(\phi(z)\right)$. 
\end{prop}
\begin{proof}
Write $\phi(z)=x$. By Lemma \ref{C3:cleft}, one sees that $\delta(a)=[x,a]$ and $\Theta=-f(x)$. Now suppose $H$ is given by some cleft data. We want to show that $H$ is a $B$-cleft extension over $A$. First of all, $H$ is a free left $A$-module with basis $\{x^i|0\le i\le p^n-1\}$, as seen in the proof of Lemma \ref{basis}. Direct computation shows that $H$ has a right $B$-comodule algebra structure determined by $x\mapsto x\otimes 1+1\otimes z$, whose coinvariant ring is $A$. Moreover, the section $\phi$ is given by $\phi(z^i)=x^i$ with convolution-inverse $\phi^{-1}(z^i)=(-)^ix^i$ for $0\le i\le p^n-1$. Now suppose $\phi(z^i)=x^i$ for $0\le i\le p^n-1$. By definition, we have
\begin{align*}
\sigma(z^{p^n-1},z)=&\sum_{i=0}^{p^n-1}{p^n-1\choose i}\left(\phi\left(z^i\right)\phi(z)\phi^{-1}(z^{p^n-1-i})+\phi\left( z^i\right)\phi(1)\phi^{-1}(z^{p^n-i})\right)\\
=&\sum_{i=0}^{p^n-1}{p^n-1\choose i}(-1)^{p^n-1-i}x^{p^n}+\sum_{i=1}^{p^n-1}{p^n-1\choose i}(-1)^{p^n-i} x^{p^n}+\phi^{-1}(z^{p^n})\\
=&(-1)^{p^n-1}x^{p^n}-\phi^{-1}\left(\lambda_{n-1}z^{p^{n-1}}+\cdots+\lambda_0z\right)\\
=&f(x).
\end{align*}
Since $A, B$ are commutative, we have $\sigma(z^{p^n-1},z)=\sigma(z,z^{p^n-1})$.  
\end{proof}

Next, we try to recover the Hopf structure of the extension. In \cite[Definition and Lemma 1.3]{Ho}, all abelian matched pairs $(H,K,\rightharpoonup,\varrho)$ form a category $\AMP$: morphisms are 
\[(\alpha,\beta): (H,K,\rightharpoonup,\varrho)\to (H',K',\rightharpoonup',\varrho'),
\]
where $\alpha: H\to H'$ is a $K$-comodule map and $\beta: K'\to K$ is an $H$-module map. We define the category $\CR$: objects are all types $T=(\mathfrak g, \mathfrak h,\rho)$ such that $\mathfrak g,\mathfrak h$ are finite abelian restricted Lie algebras with $\dim \mathfrak g=1$, and $\rho$ is an algebraic representation of $\mathfrak g$ on $\mathfrak h$; morphisms are 
\[(\alpha,\beta): (\mathfrak g,\mathfrak h,\rho)\to (\mathfrak g',\mathfrak h',\rho'),
\]
where $\alpha: \mathfrak g\to \mathfrak g'$ and $\beta: \mathfrak h'\to \mathfrak h$ are restricted Lie algebra maps satisfying the following commutative diagram
\begin{align}\label{C3:CommD}
\xymatrix{
\mathfrak g\otimes \mathfrak h\ar[rr]^-{\rho}&&\mathfrak h\\
\mathfrak g\otimes \mathfrak h'\ar[r]_-{\alpha\otimes \Id}\ar[u]^-{\Id\otimes \beta} &\mathfrak g'\otimes \mathfrak h'\ar[r]_-{\rho'}&\mathfrak h'\ar[u]_-{\beta}.
}
\end{align}

\begin{lem}\label{C4:Apair}
Suppose $\dim B=p$. Then for any associated abelian pair $(\rightharpoonup,\varrho)$, we have:
\begin{itemize}
\item[(a)] $\varrho$ is trivial.
\item[(b)] $\rightharpoonup$ is induced by an algebraic representation $\rho$ of $\mathfrak g$ on $\mathfrak h$.
\end{itemize}
\end{lem}
\begin{proof}
(a) By definition, $(B,\varrho)$ is an $A$-comodule coalgebra. Then it is clear that $\Prim(B)$ is $A$-costable under $\varrho$. Since $\dim\mathfrak g=1$, one sees that $\varrho(z)=z\otimes g$ for some group-like element $g$ in $A$. Therefore, $g=1$ for $A$ is connected. Hence $A$ coacts trivially on $\mathfrak g$. Still by \cite[Definition 2.2]{MK2}, the coinvariant $B^{coA}$ is a subalgebra of $B$. Now $B$ is generated by $z$, hence the coaction $\varrho$ is trivial.

(b) Suppose $\varrho$ is trivial. By \cite[Definition 2.2]{MK2}, one sees that $\Delta(a\rightharpoonup t)=\sum (a_1\rightharpoonup t_1)\otimes (a_2\rightharpoonup t_2)$. Moreover, $\e(a\rightharpoonup t)=\e(a)\e(t)$ for all $a\in B$ and $t\in A$ by \cite[Lemma 1.2]{Tak}. Hence, $\rightharpoonup$ makes $A$ into a $B$-module Hopf algebra. The result follows from Proposition \ref{eqdef}.
\end{proof}

\begin{cor}\label{C4:subcat}
The category $\CR$ is a full subcategory of $\AMP$ arising from $B$-extensions over $A$ satisfying $\dim B=p$.
\end{cor}
\begin{proof}
By Lemma \ref{C4:Apair}, abelian matched pairs associated with $B$-extensions over $A$ are bijective to all types in $\CR$ with trivial coaction. Moreover, the category of restricted Lie algebras with restricted Lie algebra maps is isomorphic to the category of restricted universal enveloping algebras with Hopf algebra maps. Then it is direct to check $\CR$ is a full subcategory of $\AMP$.
\end{proof}
From now on, we let $\dim B=p$. Then every abelian matched pair $(B, A, \rightharpoonup,\varrho)$ is determined by some type $T=(\mathfrak g,\mathfrak h,\rho)$. We also say the $B$-extension over $A$ is of type $T$. Finally, we are able to prove the inverse statement of Theorem \ref{C0:thmA} (b) providing that $\dim B=p$.

\begin{prop}\label{C4:TExt}
Any Hopf algebra extension of type $T$ is given by some data $\mathscr D$.
\end{prop}
\begin{proof}
Let $T=(\mathfrak g,\mathfrak h,\rightharpoonup)$. By Proposition \ref{C3:cleftD} and Lemma \ref{C4:Apair}, the algebra structure of the extension is given by some cleft data $(\rightharpoonup,\Theta)$ satisfying $z\rightharpoonup \Theta=0$. Regrading the coalgebra structure, it is clear from \cite[Proposition 3.7 (b)]{Ho} and the trivial coaction that $\Delta(x)=x\otimes 1+1\otimes x+\chi$ for some $\chi\in A^+\otimes A^+$. Moreover, $\chi$ is a $2$-cocycle in the cobar construction on $A$, as is seen from the proof of \cite[Theorem 7.5, P. 21]{Zh1}. Note that the algebra map $\Delta: H\to H\otimes H$ preserves the relation $f(x)=-\Theta$.
Then by Lemma \ref{AppOperator}, we have
\begin{align*}
-\Delta(\Theta)=&\Delta\left(f(x)\right)\\
=&f\left(x\otimes 1+1\otimes x+\chi \right)\\
=&f(x)\otimes 1+1\otimes f(x)+\RT_z^{1}(\chi)+\lambda \chi\\
=&\ZO_z(\chi)-(\Theta\otimes 1+1\otimes \Theta).
\end{align*}
It is equivalent to $\ZO_z(\chi)=d^1(\Theta)$. Set $\mathscr D=(T,z,\chi,\Theta)$. One sees that $H$ must be a quotient of $u(\mathscr D)$. Then the result follows from the dimension argument since $u(\mathscr D)$ is already an extension of $A$ by $B$.
\end{proof}

\begin{cor}\label{C4:ObjH}
Any Hopf algebra in $\CH$ is given by some data $\mathscr D$.
\end{cor}
\begin{proof}
It suffices to show $A:=u(\Prim(H))$ is normal in $H$. If so, the $p$-dimensional connected quotient $H/A^+H$ \cite[Proposition 2.2(3)]{Wan} must be isomorphic to some $B$. Then it follows from \cite[Lemma-Definition 1.1]{NsubHA}. Because $A$ has $p$-index one in $H$ \cite[Definition 2.3]{Wan}, it is clear that $A\subset H$ is a level one inclusion of some degree $n$. Moreover, the $n$-th coradical filtration of $H$ is spanned by $A_n$ and some other element $x$ by \cite[Lemma 4.1]{Wan}. As seen in the proof of \cite[Theorem 7.5, P. 21]{Zh1}, we know $\Delta(x)-x\otimes 1-x\otimes 1\in A\otimes A$. Direct computation shows that $[A,x]\subset \Prim(H)\subset A$ for $A$ is commutative. Hence the normality follows from \cite[Lemma 4.2]{Wan}. 
\end{proof}

Since our goal is to classify $B$-extensions over $A$, we give an explicit formula for all possible isomorphisms between them. Fix a type $T\in \CR$. Any $g$ in $\Aut(T)$ contains a pair of automorphisms of $\mathfrak g$ and $\mathfrak h$ satisfying the commutative diagram \eqref{C3:CommD}. We will keep the same notation $g$ for all these automorphisms. The following lemma is in the same spirit of \cite[Lemma 2.19]{Ma1} regarding a direct method approach.

\begin{lem}\label{C4:IsomMap}
Let $\mathscr D',\mathscr D$ be two data for $B$-extensions over $A$ of a fixed type $T$.
\begin{itemize}
\item[(a)] Let $F: u(\mathscr D')\to u(\mathscr D)$ be a Hopf algebra isomorphism satisfying $F(A)\subseteq A$. Then there is a unique pair $(t,g)\in A^+\times \Aut(T)$ such that
\begin{align}\label{C3:Isomap}
F(a)=g^{-1}(a),\ F(x')=\gamma(x+t)
\end{align}
for all $a\in A$ and $\gamma\in k^\times$ is the scalar such that $g(z)=\gamma z$.
\item[(b)] Let $(t,g)\in A^+\times \Aut(T)$. Then the map $F: u(\mathscr D')\to u(\mathscr D)$ determined by \eqref{C3:Isomap} is well-defined, if and only if
\begin{align}\label{C3:IsoCon}
\ZO_z(t)=\Theta-\gamma^{-p}g^{-1}(\Theta'),\ d^1(t)=\chi-\gamma^{-1}\left(g^{-1}\otimes g^{-1}\right)(\chi').
\end{align}
In this case, $F$ is a Hopf algebra isomorphism.
\item[(c)] Let $F': u(\mathscr D'')\to u(\mathscr D'),\ F: u(\mathscr D')\to u(\mathscr D)$ be isomorphisms of Hopf algebras determined as in (a) by $(t',g'), (t,g)\in A^+\times \Aut(T)$, respectively. Then the composition $FF'$ is determined by $(t+\gamma^{-1}g^{-1}(t'), g'g)$.
\end{itemize}
\end{lem}
\begin{proof}
(a) Consider the following commutative diagram:
\[
\xymatrix{
1\ar[rr]
&&A\ar[rr]\ar[d]_-{F|_A}
&&u(\mathscr D')\ar[rr]\ar[d]_-{F}
&&B\ar[rr]\ar[d]_-{\overline{F}}  \ar@/_/@{-->}[ll]_-{\phi'}
&& 1\\
1\ar[rr]
&&A\ar[rr]
&&u(\mathscr D)\ar[rr]
&&B\ar[rr]\ar@/_/@{-->}[ll]_-{\phi}
&&1.
}
\]
By the definition of abelian matched pair associated to a $B$-extension over $A$ \cite[Definition and Lemma 3.5]{Ho}, it is easy to check the morphism $\left(\overline{F},F^{-1}|_A\right)\in \Aut(T)$. Write $g=\left(\overline{F},F^{-1}|_A\right)$. Since $F\phi'\overline{F^{-1}}$ and $\phi$ are sections $B\to u(\mathscr D)$, by \cite[Lemma 1.1]{Ma1}, there is a unitary, invertible linear map $\psi: B\to A$ such that $F\phi'\overline{F^{-1}}=\psi*\phi$. Set $t=\psi(z)\in A$, and $\gamma\in k^\times$ such that $\overline{F}(z)=\gamma z$. Then one sees easily that \eqref{C3:Isomap} holds. The uniqueness is clear.

(b) It is straightforward that $F$ is a well-defined algebra map if and only if $F\left(f(x')\right)=-F(\Theta')$ and $F\left([x',a]\right)=F(z\rightharpoonup a)$ for all $a\in A$. Similarly, $F$ is a well-defined coalgebra map if and only if $(F\otimes F)\Delta(x')=\Delta\left(F(x')\right)$. Then the conditions are verified directly. Moreover, a bialgebra map $F: u(\mathscr D')\to u(\mathscr D)$ is automatically a Hopf algebra map by \cite[Proposition 4.2.5]{IntroHA}. Finally, $F$ is certainly invertible. 

(c) Easy.
\end{proof}

\section{Cohomology groups}\label{S:Grp}
In this section, we still denote $A=u(\mathfrak h)$ with $\mathfrak h$ abelian. We further assume that $B=u(\mathfrak g)$, where $\dim \mathfrak g=1$ with $z^p+\lambda z=0$ for some $\lambda\in k$. Note that any Hopf algebra extension of $A$ by $B$ is given by some data $\mathscr D=(T,z,\Theta,\chi)$. We define a cohomology group $\HG^2(B,A)$, which classifies all extensions of a fixed type $T$ up to equivalence. Also we define a subset $\HG^2(T)$ of $\HG^2(B,A)$ representing those extensions from $\CH$. As an application of Lemma \ref{C4:IsomMap}, we have an $\Aut(T)$-action on $\HG^2(T)$, whose orbits correspond to isomorphism classes in $\CH$ of the fixed type $T$.

\begin{definition}\label{D:HG}
In the set $A^+\times Z^2(\Cobar A)$, we define
\begin{itemize}
\item[(a)] a subset $\EL^2(B,A)$, where $(\Theta,\chi)\in \EL^2(B,A)$ if and only if
\[\rho_z(\Theta)=0,\ \ZO_z(\chi)=d^1(\Theta);\]
\item[(b)] an equivalence relation $\sim$, where $(\Theta',\chi')\sim (\Theta, \chi)$ if and only if there exists some $a\in A^+$ such that 
\[\Theta'-\Theta=\ZO_z(a),\ \chi'-\chi=d^1(a).\] 
\end{itemize}
\end{definition}

It is clear that $\sim$ is well-defined. Note that $A^+\times Z^2(\Cobar A)$ becomes an abelian group via componentwise addition. 
\begin{lem}
The subset $\EL^2(B,A)$ is a subgroup of $A^+\times Z^2(\Cobar A)$, and it is $\sim$-invariant.
\end{lem}
\begin{proof}
The subgroup part is easy to check. We only show that $\EL^2(B,A)$ is $\sim$-invariant.
Let $(\Theta,\chi)\in \EL^2(B,A)$. Suppose it is equivalent to some $(\Theta',\chi')\in A^+\times Z^2(\Cobar A)$.
We need to show $(\Theta',\chi')\in \EL^2(B,A)$. 
By Definition \ref{D:HG} (b), there is some $a\in A^+$ such that 
\[\Theta'=\Theta+\ZO_z(a),\ \chi'=\chi+d^1(a).\]
Thus, we have $\ZO_z(\chi')=\ZO_z(\chi)+\ZO_z(d^1(a))=d^1(\Theta+\ZO_z(a))=d^1(\Theta')$. By Proposition \ref{Maps} (c), $\rho_z(\Theta')=\rho_z(\Theta+\ZO_z(a))=\rho_z(\Theta)+\rho_z\circ\ZO_z(a)=0$. Done.
\end{proof}

In order to identify those extensions from $\CH$ in $\EL^2(B,A)$, we consider the complementary situation, when the extensions are exactly primitively generated. This suggests us to give the following definition.
\begin{definition}
In the set $A^+\times \cobo^2(\Cobar A)$, we define a subset $\LL^2(\mathfrak g,\mathfrak h)$, where $(\Theta,\chi)\in \LL^2(\mathfrak g,\mathfrak h)$ if and only if 
\[\rho_z(\Theta)=0,\ \ZO_z(\chi)=d^1(\Theta).\]
\end{definition}

We can view $\LL^2(\mathfrak g,\mathfrak h)$ as a subset of $\EL^2(B,A)$ via the inclusion $\cobo^2(\Cobar A)\subseteq Z^2(\Cobar A)$. It is straightforward to verify that it is also $\sim$-invariant. Hence we can form the following cohomology groups with respect to the equivalence relation.
\begin{definition}
We define
\begin{itemize}
\item[(a)] $\HG^2(B,A):=\EL^2(B,A)/\sim$,
\item[(b)] $\HG^2(\mathfrak g,\mathfrak h):=\LL^2(\mathfrak g,\mathfrak h)/\sim$,
\item[(c)] $\HG^2(T):=\HG^2(B,A)\setminus \HG^2(\mathfrak g,\mathfrak h)$.
\end{itemize}
\end{definition}

Note that $\HG^2(B,A)$ is an abelian group with subgroup $\HG^2(\mathfrak g,\mathfrak h)$, but $\HG^2(T)$ generally is not an abelian group.
Furthermore, in order to classify $\CH$ of a fixed type $T$, we need to take the automorphism group of $T$ into consideration. 
Choose any $g\in \Aut(T)$. There exists a group character $\gamma: \Aut(T)\to k^\times$ such that $\gamma_g$ is given by $g(z)=\gamma_g(z)$. Based on Lemma \ref{C4:IsomMap}, we define the $\Aut(T)$-action on $A^+\times \cocy^2(\Cobar A)$ by
\begin{align}\label{G-action}
g.(\Theta,\chi):=\big(\gamma_g^{p}g(\Theta),\gamma_g(g\otimes g)(\chi)\big).
\end{align}

We claim that the action is well-defined. However, choose any point $(\Theta,\chi)\in A^+\times \cocy^2(\Cobar A)$. Since $\Delta$ is an algebra map, we have $\overline{\Delta}g=(g\otimes g)\overline{\Delta}$ on $A^+$. Thus,
\[\diff^2(g.\chi)=(-\overline{\Delta}\otimes 1+1\otimes \overline{\Delta})\left(\gamma_g(g\otimes g)(\chi)\right)=\gamma_g(g\otimes g\otimes g)\left(\diff^2(\chi)\right)=0.\]
So $g.\chi\in \cocy^2(\Omega A)$, and certainly $g.\Theta\in A^+$. Then it is direct to check that it is a group action. Moreover, one sees that $\EL^2(B,A)$ and $\LL^2(\mathfrak g,\mathfrak h)$ are both $\Aut(T)$-invariant and $\sim$ is preserved under the action. As a consequence, we have an induced $\Aut(T)$-action on $\HG^2(T)=\HG^2(B,A)\setminus \HG^2(\mathfrak g,\mathfrak h)$.
\newline

\noindent
\textbf{Proof of Theorem \ref{C0:thmB}.} By Proposition \ref{C4:TExt}, any $B$-extension over $A$ is given by some data $\mathscr D$. Therefore, there is a bijection between $B$-extensions over $A$ and elements of $\EL^2(B,A)$ of a fixed type.

(a) Any restricted Lie algebra extension of $\mathfrak h$ by $\mathfrak g$ gives a primitively generated extension of $A$ by $B$ via its restricted universal enveloping algebra, and vice visa. Hence, equivalence classes of restricted Lie algebra extensions correspond to a subgroup of $\HG^2(B,A)$ characterized by primitively generated extensions. 
By Theorem \ref{C0:thmA} (c), the subgroup consists of all points $(\chi,\Theta)\in \EL^2(B,A)$ satisfying $[\chi]=0$ in $\HL^2(\Cobar A)$, or equivalently $\chi\in \cobo^2(\Cobar A)$. Then the result follows from the definition.

(b) It is a consequence of Lemma \ref{C4:IsomMap}, where we let $g=\Id$.
\newline

\noindent
\textbf{Proof of Theorem \ref{C0:thmC}.} By definition, any $B$-extension over $A$ belongs to $\CH$ if and only if it is not primitively generated. Hence, there is a bijection between $B$-extensions over $A$ belonging to $\CH$ and elements of $\EL^2(B,A)\setminus \LL^2(\mathfrak g,\mathfrak h)$ of a fixed type $T$. Then the first part is similar to Theorem \ref{C0:thmB} (b). Next, any Hopf algebra in $\CH$ is a $B$-extension over $A$ of some type $T$ by Corollary \ref{C4:ObjH}. Furthermore, any isomorphism between two objects of $\CH$ must induce a commutative diagram between their extensions for the Hopf algebra map preserves primitive elements. Therefore, two isomorphic Hopf algebras in $\CH$ have isomorphic types as seen in the proof of Lemma \ref{C4:IsomMap}. So it suffices to consider all non-isomorphic types to classify $\CH$.
\newline

Finally, we want to compare our cohomology group $\HG^2(B,A)$ with the cohomology group $\HL^2(B,A)$ defined in \cite{Ho, MK2}. Let $H$ be a $B$-extension over $A$ associated with an abelian matched pair $(\rightharpoonup,\varrho)$. Then we have a Sweedler $2$-cocycle and a dual $2$-cocycle (depending on the chosen $H\cong A\otimes B)$
\[\sigma: B\otimes B\to A,\ \tau: B\to A\otimes A\]
such that $\rightharpoonup$ and $\sigma$ construct on $H=A\#_\sigma B$ a $B$-crossed product which determines a right $B$-comodule algebra structure on $H$, while $\varrho$ and $\tau$ construct on $H=A \!^\tau\#B$ an $A$-crossed coproduct which determines a left $A$-module coalgebra structure on $H$. The pair $(\sigma,\tau)$ gives a total $2$-cocycle of a certain double complex, and $(H)\mapsto (\sigma,\tau)$ gives an isomorphism $\Opext(B,A,\rightharpoonup,\varrho)\cong \HL^2(B,A)$; see \cite[Section 2]{Ho}. 
\begin{prop}\label{C5:IsomGrp}
For $B$-extensions over $A$ of a fixed type, we have the following isomorphisms between cohomology groups:
\begin{itemize}
\item[(a)] $\HL^1(B,A)\cong \Aut\left(\mathfrak h\rtimes \mathfrak g\right)\cong \Ker (\ZO_z: \mathfrak h\to \mathfrak h)$.
\item[(b)] $\HL^2(B,A)\cong \HG^2(B,A)$.
\end{itemize}
\end{prop}
\begin{proof}
(a) By \cite[Proposition 6.5]{Ho}, we know $\HL^1(B,A)\cong \Aut(A\#B)\cong \Aut(\mathfrak h\rtimes \mathfrak g)$. By \cite[Definition 6.1]{Ho}, any $f\in \Aut(\mathfrak h\rtimes \mathfrak g)$ is given by some $a\in \mathfrak h$ such that $f(z)=z+a$. Then it is a direct check.

(b) Define map $F: \Opext(B,A)\to \HG^2(B,A)$ induced by $(\sigma,\tau)\mapsto \left(-\sigma(z^{p-1},z),\tau(z)\right)$. By Proposition \ref{C3:cleftD}, one sees that $\left(-\sigma(z^{p-1},z),\tau(z)\right)=(\Theta,\chi)$ in the data $\mathscr D$. Note that the addition in $\Opext(B,A)$ is induced by the convolution product. Since $z$ is primitive in $B$, it is straightforward to check that $F$ preserves the addition. It is clear that $F$ has a natural inverse. Then the result follows from Theorem \ref{C0:thmB} (b).
\end{proof}

\begin{ques}
In general case when $\dim B=p^n$ for some $n\ge 1$, how to present  $B$-extensions over $A$ in generators and relations and describe the cohomology group $\HL^2(B,A)$ concretely?
\end{ques}

\section{A geometric realization}\label{S:Realization}
In this section, let $k$ be a perfect field of characteristic $p>0$. We still denote $A=u(\mathfrak h)$ with $\mathfrak h$ abelian, and assume that $B=u(\mathfrak g)$, where $\dim \mathfrak g=1$ with $z^p+\lambda z=0$ for some $\lambda\in k$. Let $T=(\mathfrak g,\mathfrak h,\rho)$ be some type in $\CR$. In the following, we fix a basis $\{x_1,x_2,\cdots,x_d\}$ for $\mathfrak h$. It is clear that $\HL^1(\Cobar A)=\Prim(A)=\mathfrak h$.
We define a map $\Bock: \HL^1(\Cobar A)\to \HL^2(\Cobar A)$ as
\begin{align*}\label{Bockhom}
\Bock(x)=\left[\sum_{1\le i\le p-1}{p\choose i}/ p\ x^i\otimes x^{p-i}\right],
\end{align*}
for any $x\in \mathfrak h$. Since $A$ is commutative, one sees that $\Bock$ is semi-linear with respect to the Frobenius map of $k$. By abuse of language, we also consider $\Bock(x)=\sum_{i=1}^{p-1} {p\choose i}/p\ x^i\otimes x^{p-i}$ as an element in $(A^+)^2$ for any $x\in A^+$. Since $\HL^\bullet\left(\Cobar A\right)\cong \HL\HL^\bullet(A^*,k)\cong \HL^\bullet(C_p^d,k)$, we have
\begin{align*}
\HL^\bullet\left(\Cobar A \right)\cong
\begin{cases}
S(\mathfrak h) &p=2\\
\Lambda(\mathfrak h)\otimes S(\Bock(\mathfrak h))&p>2
\end{cases}
\end{align*}
where $S(\mathfrak h)$ and $\Lambda(\mathfrak h)$ are the polynomial and exterior algebras over $\mathfrak h$. In particular, we have the vector space isomorphism 
\begin{equation}\label{E:HG2}
\HL^2(\Cobar A)\cong
\begin{cases}
S^2(\mathfrak h)& p=2\\
\Lambda^2(\mathfrak h)\oplus \Bock(\mathfrak h)& p>2.
\end{cases}
\end{equation}
See references \cite[Section 4]{QAST} and \cite[Proposition 6.2]{Wan}.

Our approach is to view $\HG^2(T)$ as a subquotient space of $\mathbb A^d\times\mathbb A^{d(d+1)/2}$ via some embedding. We give the construction for characteristic $p>2$, and it is similar for $p=2$ by \eqref{E:HG2}. First of all, we embed the affine space $\mathbb A^d\times\mathbb A^{d(d+1)/2}$ into $\mathfrak h\times \cocy^2(\Cobar A)$ by sending any point 
\[P=(a_i,b_{jk},c_l)_{1\le j<k\le d,1\le i,l\le d}\] 
to some $(\Theta_P,\chi_P)$ such that 
\[
\Theta_P=\sum_{1\le i\le d}a_ix_i,\ \chi_P=\sum_{1\le i<j\le d} b_{ij}x_i\otimes x_j+\Bock\left(\sum_{1\le i\le d}c_ix_i\right).
\]
Note that $A=u(\mathfrak h)$ has a PBW basis $\{x_1^{\sigma_1}x_2^{\sigma_2}\cdots x_d^{\sigma_d}|0\le \sigma_1,\sigma_2,\cdots,\sigma_d\le p-1\}$, where we denote by $A_{\ge 2}$ the subspace of $A$ spanned by all these bases satisfying $\sigma_1+\sigma_2+\cdots+\sigma_d\ge 2$. Thus, there is a vector space decomposition $A^+=A_{\ge 2}\oplus \mathfrak h$. 

Next, we define the subset $S_T$ of $\mathbb A^d\times \mathbb A^{d(d+1)/2}$ such that $P\in S_T$ if and only if   
\begin{itemize}
\item[(a)] $\chi_P\not\in \cobo^2(\Cobar A)$,
\item[(b)] $\ZO_z(\chi_P)=d^1(a)$,
\item[(c)] $\rho_z(a+\Theta_P)=0$,
\end{itemize}
for some $a\in A_{\ge 2}$. The element $a\in A_{\ge 2}$ is uniquely determined by $\chi_P$, which is easy to see by taking the difference of two possible candidates. So we can write $a$ by $\Psi_P$ for any $P\in S_T$. 

Finally, identify $A^d$ with $\mathfrak h$, and let $\mathfrak h$ act on $A_k^d$ by addition via $\ZO_z$, i.e., $\Theta.x:=x+\ZO_z(\Theta)$ for all $x\in \mathbb A^d$ and $\Theta\in \mathfrak h$. The resulting quotient space is denoted by $\mathbb A^d/\mathfrak h$. The quotient map $\pi$ is defined to be 
\begin{align*}
\xymatrix{
\mathbb A^d\times \mathbb A^{d(d+1)/2}\ar@{->>}[r]^-{\pi}& (\mathbb A^{d}/\mathfrak h)\times \mathbb A^{d(d+1)/2}.
}
\end{align*}

\begin{prop}\label{P:Embed}
Every equivalence class in $\EL^2(B,A)\setminus \LL^2(\mathfrak g,\mathfrak h)$ can be represented by $(\Psi_P+\Theta_P,\chi_P)$ for some $P\in S_T$. Moreover, elements of $\mathcal H^2(T)$ are in $1$-$1$ correspondence with points in $\pi(S_T)$.
\end{prop}
\begin{proof}
Let $(\Theta,\chi)\in \EL^2(B,A)\setminus \LL^2(\mathfrak g,\mathfrak h)$.
By definition, we know $\chi$ is a $z$-cocycle. Thus, we can write $\chi=\chi'+d^1(a)$ for some $a\in A^+$, where $\chi'=\sum_{1\le i<j\le d} b_{ij}x_i\otimes x_j+\Bock(\sum_{1\le i\le d}c_ix_i)$ according to \eqref{E:HG2}. Let $\Theta'=\Theta-\ZO_z(a)$. One sees that $(\Theta',\chi')\sim (\Theta,\chi)$, and $(\chi',\Theta')\in \EL^2(B,A)\setminus \LL^2(\mathfrak g,\mathfrak h)$. Now, we show that $(\Theta',\chi')$ is given by some point in $S_T$.
By the vector space decomposition $A^+=A_{\ge 2}\oplus \mathfrak h$, we can write $\Theta'=\Theta'_2+\Theta'_1$, where $\Theta'_2\in A_{\ge 2}$ and $\Theta'_1=\sum_{1\le i\le d} a_ix_i$. Let $P=(a_i,b_{jk},c_l)\in \mathbb A^d\times \mathbb A^{d(d+1)/2}$. By definition, it is clear that $(\Theta_P,\chi_P)=(\Theta'_1,\chi')$ and $P\in S_T$ by taking $\Psi_P=\Theta_2'$. 

Next, we define a map $r: S_T\to \mathcal H^2(T)$ by $r(P)=\left[(\Psi_P+\Theta_P,\chi_P)\right]$. One sees that $r$ is a well-defined surjective map by previous discussion. Hence, it remains to show that $r$ factors through the quotient map $\pi$ and the factorization is injective. Let $P,Q$ be two points in $S_T$. It suffices to prove $r(P)=r(Q)$ if and only if $\pi(P)=\pi(Q)$.
By Definition \ref{D:HG} (b), $r(P)=r(Q)$ if and only if there exists some $a\in A^+$ such that
\[\ZO_z(a)=(\Psi_P+\Theta_P)-(\Psi_Q+\Theta_Q),\ d^1(a)=\chi_P-\chi_Q.\]
Note that $d^1(a)=\chi_P-\chi_Q$ implies that $[\chi_P]=[\chi_Q]$ in $\HL^2(\Cobar A)$. Thus, $\chi_P=\chi_Q$ by \eqref{E:HG2}, and hence $\Psi_P=\Psi_Q$.
Therefore, we have $r(P)=r(Q)$ if and only if $\chi_P=\chi_Q$ and there exists some $a\in \mathfrak h$ ($d^1(a)=0$) such that $\ZO_z(a)=\Theta_P-\Theta_Q$. In other words, $\Theta_P$ and $\Theta_Q$ are in the same orbit of the $\mathfrak h$-action on $\mathbb A^d$.
\end{proof}

\begin{ques}
Is $\HG^2(T)$ always isomorphic to some open closed subset of $\mathbb A^n$ for some $n\ge 1$?
\end{ques}

Another way to understand $\mathcal H^2(T)$ is to consider the following commutative diagram
\[
\xymatrix{
 \mathcal H^2(T)\ar[d]^-{q}\ar@{^{(}->}[r]&(\mathbb A^{d}/\mathfrak h)\times \mathbb A^{d(d+1)/2}\ar[d]^-{p_2}\\
  \HL^2(\Cobar A)\ar@{=}[r]   &\mathbb A^{d(d+1)/2}.
}
\]
The above projection $q: [(\Theta,\chi)]\mapsto [\chi]$ sends every equivalence class to some nonzero $z$-characteristic element in $\HL^2(\Cobar A)$. 

\begin{definition}\label{D:Ad}
Let $\xi$ be $z$-characteristic in $\HL^2(\Cobar A)$, which is represented by some cocycle $\chi$. 
We say $\xi$ is \emph{admissible} if there exists some $a\in A^+$ such that $\ZO_z(\chi)=d^1(a)$ and $\rho_z(a)=0$. We denote by $\mathscr A^2(\Cobar A)$ all the admissible elements in $\HL^2(\Cobar A)$.
\end{definition}

Wee need to show that admissibility does not depend on the choice of the cocycle $\chi$. 
Let $\chi'$ be another cocycle representing $\xi$. Then there is some $X\in A^+$ such that $\chi'=\chi+d^1(X)$. Suppose there is some $a\in A^+$ satisfying $\ZO_z(\chi)=d^1(a)$ and $\rho_z(a)=0$. Let $b=a+\ZO_z(X)$. It is easy to see that $\ZO_z(\chi')=d^1(b)$ and $\rho_z(b)=\rho_z(a)+\rho_z\circ \ZO_z(X)=0$ by Proposition \ref{Maps} (c). 

\begin{remark}\label{R:Ad}
We make some observations concerning admissible cocycles.
\begin{itemize}
\item[(a)] There is a surjective factorization $q: \HG^2(T)\twoheadrightarrow \mathscr A^2(\Cobar A)\setminus \{0\}$.
\item[(b)] Let $\chi$ be a $z$-cocycle, and $a\in A_{\ge 2}$ the unique element determined by $\ZO_a(\chi)=d^1(a)$. Then $[\chi]$ is admissible if and only if $\rho_z(a)\in \Img\rho_z$.
\item[(c)] The admissibility is preserved by base field extension.
\end{itemize}
\end{remark}

We point out that not all $z$-characteristic elements are admissible. Suppose $\mathfrak g^p=0$, and $\mathfrak h$ is spanned by $\{x_1,x_2,x_3\}$ such that $x_1^p=x_3^p=0,x_2^p=x_3$. If $\rho_z(x_1)=x_2,\rho_z(x_2)=\rho_z(x_3)=0$, then $[\omega(x_1)]$ is $z$-characteristic but not admissible. But they coincide in some circumstances, which is our next result. Recall in \cite{Strade}, a finite-dimensional restricted Lie algebra $L$ is said to be a \emph{torus} if it is abelian and every element of $L$ is semisimple in $u(L)$, i.e., generates a semisimple subalgebra of $u(L)$.

\begin{prop}\label{P:TAd}
If either $\mathfrak h$ or $\mathfrak g$ is a torus, then every $z$-characteristic element is admissible. 
\end{prop}
\begin{proof}
We can assume $k$ is algebraically closed. We only treat the case when characteristic $p=2$. The argument is similar for $p>2$. Let $\xi$ be a $z$-characteristic element in $\HL^2(\Cobar A)$. Since the admissibility does not depend on the choice of the representing cocycle for $\xi$, we can write the cocycle as
\begin{align*}
\chi=\sum_{1\le i\le j\le d} \mu_{ij}x_i\otimes x_j
\end{align*}
for some coefficients $\mu_{ij}\in k$ according to \eqref{E:HG2}.

(a) $\mathfrak h$ is a torus. By \cite{Hoch}, we can further assume that $x_i^p=x_i$ for $1\le i\le d$. Hence, by Definition \ref{AR} (d), we know $\rho_z=0$. Therefore,
\begin{align*}
\ZO_z(\chi)=\chi^p+\lambda\chi+\rho_z^{p-1}(\chi)=\sum_{1\le i\le j\le d} (\mu_{ij}^p+\lambda \mu_{ij})x_i\otimes x_j.
\end{align*}
Since $\{[x_i\otimes x_j]|1\le i\le j\le d\}$ is a basis for $\HL^2(\Cobar A)$, all the coefficients $\mu_{ij}^p+\lambda \mu_{ij}=0$. Therefore $\ZO_z(\chi)=0$, which implies that $\xi$ is admissible.

(b) $\mathfrak g$ is a torus. Then, in the relation $z^p+\lambda z=0$, we have $\lambda\neq 0$. By Definition \ref{AR} (d), we know $\rho_z(\mathfrak h^p)=0$. Since $k$ is algebraically closed, $\mathfrak h^p$ is a subspace of $\mathfrak h$, and $\rho_z$ is diagonalizable on $\mathfrak h/\mathfrak h^p$. Without loss of generality, we can assume in the basis $x_1,x_2,\cdots,x_d$ of $\mathfrak h$, the first $s$ elements form a basis for the subspace $\mathfrak h^p$, and the images of the remaining $d-s$ elements are eigenvectors in the quotient space $\mathfrak h/\mathfrak h^p$. In other words, we can write $\rho_z(x_i)=0$ for all $1\le i\le s$, and $\rho_z(x_j)=\sigma_jx_j+y_j$ for some $y_j\in \mathfrak h^p$ for $s+1\le j\le d$. It is easy to see that, if the eigenvalue $\sigma_j=0$, we have $y_j$=0. 
Therefore, by replacing $x_j$ by $x_j+y_j/\sigma_j$ when $\sigma_j\neq 0$, we have $\rho_z(x_i)=\sigma_ix_i$ for all $1 \le i\le d$.

Now, we consider the Hopf subalgebra of $A$, which is generated by the restricted Lie subalgebra $\mathfrak h^p$. We denote it by $C=u(\mathfrak h^p)$. Direct computation shows that
\begin{align}\label{P:EAd}
\ZO_z(\chi)=&\ \chi^p+\sum_{1\le i\le j\le d}\lambda\mu_{ij}x_i\otimes x_j+\sum_{\stackrel{k+l=p-1}{1\le i\le j\le d}}{p-1\choose k}\mu_{ij}\rho_z^k(x_i)\otimes  \rho_z^{l}(x_j)\\
=&\ \chi^p+\sum_{1\le i\le j\le d} \mu_{ij}\left(\lambda+(\sigma_i+\sigma_j)^{p-1}\right)x_i\otimes x_j.\notag
\end{align}
Since the $p$-th map in $\Cobar A$ commutes with the differentials in $\Cobar A$ by Proposition \ref{Maps} (a), we can view $\chi^p$ as a cocycle in the subcomplex $\Cobar C$. Therefore, there exists some $X\in C^+$ such that
\[
\chi^p=d^1(X)+\sum_{1\le i\le j\le s} \tau_{ij}x_i\otimes x_j.
\]
Combine it with above \eqref{P:EAd}, we get
\[
\ZO_z(\chi)=d^1(X)+\sum_{1\le i\le j\le d} \phi_{ij}x_i\otimes x_j,
\]
for some new coefficients $\phi_{ij}\in k$. Hence, $\phi_{ij}=0$ by the same reason as in (a). Then $\ZO_z(\chi)=d^1(X)$ and $\rho_z(X)=0$ for $X\in A^p$, which implies that $\xi$ is admissible.
\end{proof}

\begin{ques}
Is there a admissibility criterion for $z$-characteristic elements?
\end{ques}

Let $\xi$ be any nonzero admissible element in $\HL^2(\Cobar A)$. We write $\Ker \rho_z:=\Ker(\rho_z:\mathfrak h\to \mathfrak h)$. Because of Proposition \ref{Maps} (c), the $\mathfrak h$-action can be restricted to the subspace $\Ker\rho_z$. We want to study the fiber of the projection $q$. 
 
\begin{lem}\label{L:Fiber}
Points in the fiber $q^{-1}(\xi)$ are in $1$-$1$ correspondence with $\mathfrak h$-orbits in $\Ker \rho_z$. 
Moreover, the map $q$ is injective if and only if $\Ker\rho_z=\Img\ZO_z$.
\end{lem}
\begin{proof}
We use the notations in Proposition \ref{P:Embed}. Consider the following commutative diagram:
\[
\xymatrix{
S_T\ar[rr]^-{\pi}\ar[dr]_-{r}  &&   (\mathbb A^d/\mathfrak h)\times \mathbb A^{d(d+1)/2}\\
&\mathcal H^2(T)\ar@{^{(}->}[ur]\ar[d]_-{q}&\\
&\HL^2(\Cobar A).&
}
\]
By Proposition \ref{P:Embed},  points in $q^{-1}(\xi)$ are in $1$-$1$ correspondence with points in $\pi(rq)^{-1}(\xi)$. By Remark \ref{R:Ad} (a), there exists some point $P_0$ in the fiber $(rq)^{-1}(\xi)$. 
One sees that any point $P\in (rq)^{-1}(\xi)$ if and only if 
\[
\chi_P=\chi_{P_0},\ \rho_z(\Psi_P+\Theta_P)=0,\ \text{for some}\ \Psi_P\in A_{\ge 2}.
\]
It follows that $\Psi_P=\Psi_{P_0}$, and  $\rho_z(\Theta_P-\Theta_{P_0})=0$. Hence
\[
(rq)^{-1}(\xi)=\{P\in S_T|(\Theta_P,\chi_P)\in (\Theta_{P_0}+\Ker \rho_z,\chi_{P_0})\}.
\]
Therefore, $\pi(rq)^{-1}(\xi)=(\Theta_{P_0}+\Ker \rho_z/\mathfrak h) \times \chi_{P_0}\cong (\Ker \rho_z/\mathfrak h)\times \chi_{P_0}$. Moreover, $q$ is injective if and only if $\mathfrak h$-orbit in $\Ker\rho_z$ is single if and only if $\Ker\rho_z=\Img\ZO_z$.
\end{proof}

\begin{ques}
Can we classify all the types in $\CR$ such that the corresponding projection $q$ is injective?
\end{ques}

Note that there is an induced $\Aut(T)$-action on $\HL^2(\Cobar A)$ by \eqref{G-action}. Moreover, $\mathscr A^2(\Cobar A)$ is invariant under this induced action. 

\begin{prop}\label{P:HT=Co}
Suppose $\Ker\rho_z=\Img\ZO_z$. Then $q$ induces a bijection between $\Aut(T)$-orbits in $\HG^2(T)$ and $\Aut(T)$-orbits in $\mathscr A^2(\Cobar A)\setminus \{0\}$. Moreover, we can replace $\mathscr A^2(\Cobar A)$ by all $z$-characteristic elements if either $\mathfrak h$ or $\mathfrak g$ is a torus.
\end{prop}
\begin{proof}
By Remark \ref{R:Ad} (a), $q$ maps unto $\mathscr A^2(\Cobar A)\setminus \{0\}$. If $\Ker\rho_z=\Img\ZO_z$, then $q$ induces a bijection between $\HG^2(T)$ and  $\mathscr A^2(\Cobar A)\setminus \{0\}$ by Lemma \ref{L:Fiber}. Since $q$ is compatible with the $\Aut(T)$-action, the result follows from Proposition \ref{P:TAd}.
\end{proof}

\begin{ques}
Can we classify all the types in $\CR$ such that the set $\HG^2(T)$ is empty?
\end{ques}

Therefore, it is important to establish formulas to compute $z$-characteristic elements in $\HL^2(\Cobar A)$. Recall that any $\xi\in \HL^2(\Cobar A)$ is $z$-characteristic if $\ZO_z(\xi)=\xi^p+\lambda\xi+\rho_z^{p-1}(\xi)=0$. We work over a perfect field $k$ of characteristic $p>2$. By \eqref{E:HG2}, we can write every $\xi\in\HL^2(\Cobar A)$ as
\[
\xi=\Lambda+\Bock(x),
\]
where $\Lambda=\sum_{1\le i<j\le d} \mu_{ij}x_i\wedge x_j$ and $x=\sum_{1\le i\le d} \mu_ix_i$ for some coefficients $\mu_{ij},\mu_i$ in $k$. 
Note that the two subspaces $\Lambda^2(\mathfrak h)$ and 
$\Bock(\mathfrak h)$ are all $\ZO_z$-invariant. Hence, $\xi$ is $z$-characteristic if and only if
\[
\ZO_z(\Lambda)=0,\ \ZO_z\left(\Bock(x)\right)=0.
\]
\begin{lem}\label{L:Trivial}
For any $x\in \mathfrak h$, we have $\rho_z\left(\Bock(x)\right)=\diff^1\left(-x^{p-1}\rho_z(x)\right)$. As a consequence, the $\rho$-action is trivial on $\Bock(\mathfrak h)$.
\end{lem}
\begin{proof}
Since $\mathfrak h$ is abelian, we have $\rho_z(x^i)=ix^{i-1}\rho_z(x)$ for any $x\in \mathfrak h$. Thus,
\begin{align*}
\rho_z\left(\Bock(x)\right)=&\ \sum_{i=1}^{p-1}\frac{(p-1)!}{i!(p-i)!}\ \rho_z(x^i)\otimes x^{p-i}+\sum_{i=1}^{p-1}\frac{(p-1)!}{i!(p-i)!}\ x^i\otimes \rho_z(x^{p-i})\\
=&\ \sum_{i=1}^{p-1}\frac{(p-1)!}{(i-1)!(p-i)!}\ x^{i-1}\rho_z(x)\otimes x^{p-i}+\sum_{i=1}^{p-1}\frac{(p-1)!}{i!(p-1-i)!}\ x^i\otimes x^{p-1-i}\rho_z(x)\\
=&\ \sum_{i=0}^{p-2}{p-1\choose i}\ x^{i}\rho_z(x)\otimes x^{p-1-i}+\sum_{i=1}^{p-1}{p-1\choose i}\ x^i\otimes x^{p-1-i}\rho_z(x)\\
=&\ (x\otimes 1+1\otimes x)^{p-1}\left(\rho_z(x)\otimes 1+1\otimes \rho_z(x)\right)-\left(x^{p-1}\rho_z(x)\right)\otimes 1-1\otimes\left(x^{p-1}\rho_z(x)\right)\\
=&\ \diff^1\left(-x^{p-1}\rho_z(x)\right).
\end{align*}
\end{proof}

\begin{prop}\label{C7:zchar}
Using above notations, $\xi$ is $z$-characteristic if and only if $x^p+\lambda^{1/p}x=0$, and the following equality holds in $\Lambda^2(\mathfrak h)$:
\[
\sum_{1\le i<j\le d}\mu_{ij}^px_i^p\wedge x_j^p+\sum_{1\le i<j\le d}\lambda\mu_{ij} x_i\wedge x_j+\sum_{\stackrel{k+l=p-1}{1\le i<j\le d}}{p-1\choose k}\mu_{ij}\rho_z^k(x_i)\wedge \rho_z^{l}(x_j)=0.
\]
\end{prop}
\begin{proof}
We know $\Bock$ is semi-linear with respect to the Frobenius map of $k$. Thus, by previous lemma,
\begin{align*}
\ZO_z\left(\Lambda+\Bock(x)\right)=&\ZO_z(\Lambda)+\ZO_z[\Bock(x)]\\
=&\ZO_z(\Lambda)+\Bock(x)^p+\lambda \Bock(x)+\rho_z^{p-1}[\Bock (x)]\\
=&\ZO_z(\Lambda)+\Bock\left (x^p+\lambda^{1/p}x\right).
\end{align*}
Then the result follows from a direct computation for $\ZO_z(\Lambda)$.
\end{proof}

Finally, we consider the $\Aut(T)$-action on $\HL^2(\Cobar A)$ induced by \eqref{G-action} concerning Proposition \ref{P:HT=Co}. Let $g\in \Aut(T)$. Recall that the group character $\gamma: \Aut(T)\to k^\times$ is defined by $g(z)=\gamma_g z$. According to \eqref{E:HG2}, we can write any cohomology class $\xi\in \HL^2(\Cobar A)$ as
\[
\xi=\sum_{1\le i<j\le d}\mu_{ij}x_i\wedge x_j+\Bock\left(\sum_{1\le i\le d}\mu_ix_i\right).
\]
By \eqref{G-action}, the action of $g$ on $\HL^2(\Cobar A)$ is given by
\begin{align}\label{C8:formula}
g(\xi)=\sum_{1\le i<j\le d}\gamma_g\mu_{ij}g(x_i)\wedge g(x_j)+\Bock\left(\sum_{1\le i\le d} \sqrt[p]{\gamma_g}\mu_ig(x_i)\right).
\end{align}

Furthermore, if $k=\pfield$, we can identify $\HL^2(\Cobar A)=\Lambda^1(\mathfrak h)\oplus \Lambda^2(\mathfrak h)$. Hence we can rewrite
\[
\xi=\sum_{1\le i\le d}\mu_ix_i+\sum_{1\le i<j\le d}\mu_{ij}x_i\wedge x_j.
\]
Moreover, there is a group embedding from $\Aut(T)$ to the affine automorphism group of $\Lambda(\mathfrak h)$, which changes \eqref{C8:formula} into
\begin{align}\label{C8:format}
g(\xi)=\gamma_g\left(\sum_{1\le i\le d} \mu_ig(x_i)+\sum_{1\le i<j\le d} \mu_{ij}g(x_i)\wedge g(x_j)\right).
\end{align}
The discussion for $p=2$ is similar, so we omit it here.

%
\section{Proof of Theorem \ref{C0:thmD}}\label{S:last}
We classify all semisimple Hopf algebras in $\CH$ under the assumption that $k$ is algebraically closed of characteristic $p>0$. Let $\mathbb K$ denote the finite field of $p$ elements.
Recall the semisimplicity criteria for connected Hopf algebras.
\begin{thm*}\cite{sscHa,Wan3}\label{C9:SCH}
Let $H$ be a finite-dimensional connected Hopf algebra. The following are equivalent:
\begin{itemize}
\item[(a)] $H$ is semisimple.
\item[(b)] $H$ is commutative and semisimple.
\item[(c)] $\Prim(H)$ is a torus.
\item[(d)] $H^*\simeq k[G]$, for some $p$-group $G$.
\end{itemize} 
\end{thm*}
Let $H$ be any semisimple Hopf algebra in $\CH$. By Corollary \ref{C4:ObjH}, $H$ is a $B$-extension over $A$ such that $\Prim(H)=\mathfrak h$ is a torus. Since quotient of any semisimple Hopf algebra is semisimple, $\mathfrak g$ is a one-dimensional torus. Because $H$ is commutative, $H$ has type $T=(\mathfrak g,\mathfrak h,0)$ where $\mathfrak h,\mathfrak h$ are tori. We call type $T=(\mathfrak g,\mathfrak h,0)$ a \emph{semisimple type} if $\mathfrak g,\mathfrak h$ are tori. By Theorem \ref{C0:thmC}, isomorphism classes of semisimple Hopf algebras in $\CH$ are in $1$-$1$ correspondence with points in
\[\coprod_T\mathcal H^2(T)/\Aut(T),\] 
where $T$ runs through all non-isomorphic semisimple types in $\CR$.

First of all, we classify all semisimple types in $\CR$. Since $k$ is algebraically closed, it is clear that they are classified by $\dim \mathfrak h$. For any $d\ge 1$, we choose a representative $T_d=(\mathfrak g,\mathfrak h,0)$, where we fix a basis $z$ for $\mathfrak g$ and $x_1,x_2,\cdots,x_d$ for $\mathfrak h$ satisfying $z^p=z$ and $x_i^p=x_i$ for all $1\le i\le d$. Regarding the automorphism group, since $\rho=0$, $\Aut(T_d)=\Aut(\mathfrak g)\times \Aut(\mathfrak h)$. Direct computation shows that $\Aut(T_d)= \pfield^\times \times \GL(d,\pfield)$.

Secondly, we compute $\mathcal H^2(T_d)$ for each $T_d$. For any $x=\sum_{1\le i\le d}\mu_ix_i$ in $\mathfrak h$, it is direct to check that
\[
\ZO_z(x)=x^p-x=\sum_{1\le i\le d}(\mu_i^p-\mu_i)x_i.
\]
Hence, $\Img\ZO_z=\mathfrak h=\Ker \rho_z$ since $k=\overline{k}$. By Proposition \ref{P:HT=Co}, elements in $\mathcal H^2(T_d)$ are in $1$-$1$ correspondence with nonzero $z$-characteristic elements in $\HL^2(\Cobar A)$. Then, we follow Proposition \ref{C7:zchar} to find all the $z$-characteristic elements in $\HL^2(\Cobar A)$. When $p=2$, we can write any cohomology class by
\[
\xi=\sum_{1\le i\le j\le d}\mu_{ij}x_ix_j.
\]
Direct computation shows that
\begin{align*}
\ZO_z(\xi)=&\ \xi^p-\xi\\
=&\ \sum_{1\le i\le j\le d} \mu_{ij}^px_i^px_j^p-\sum_{1\le i\le j\le d}\mu_{ij}x_ix_j\\
=&\ \sum_{1\le i\le j\le d}(\mu_{ij}^p-\mu_{ij})x_ix_j.
\end{align*}
So $\xi$ is $z$-characteristic if and only if all $\mu_{ij}\in \pfield$. When $p>2$, we can write any cohomology class by
\[
\xi=\sum_{1\le i< j\le d}\mu_{ij}x_ix_j+\Bock\left(\sum_{1\le i\le d}\mu_ix_i\right).
\]
Similar computation yields the same result. Hence, elements in $\mathcal H^2(T_d)$ are in $1$-$1$ correspondence with points in $\mathbb A_\pfield^{d(d+1)/2}\setminus\{0\}$.

In a conclusion, to find $\HG^2(T_d)/\Aut(T_d)$, it is equivalent to consider the group action of $\pfield^\times\times\GL(d,\pfield)$ on $\mathbb A_\pfield^{d(d+1)/2}\setminus\{0\}$.
By \eqref{C8:format}, the subgroup $\pfield^\times$ acts on $\mathbb A_\pfield^{d(d+1)/2}\setminus\{0\}$ by multiplication, whose quotient space is the projective space $\mathbb P_\pfield^{d(d-1)/2}$. 
Hence in general, $\mathcal H^2(T)/\Aut(T)$ are in $1$-$1$ correspondence with
\[
\mathbb P_\pfield^{d(d-1)/2}\big/\GL(d,\pfield).
\]
\newline

\noindent
\textbf{(a) $\Longleftrightarrow$ (b) of Theorem \ref{C0:thmD}.} In the following, we work over $\pfield$. We prove the case for $p>2$ and $p=2$ is similar. We denote by $W=\wedge^1(\mathfrak h)\oplus \wedge^2(\mathfrak h)$ the degree one and two parts of the exterior algebra $\wedge(\mathfrak h)$. Points of $\mathbb P_\pfield^{d(d-1)/2}$ are in $1$-$1$ correspondence with one-dimensional subspaces of $W$ by
\[
\xymatrix{
[\mu_i,\mu_{jk}]_{1\le i\le d,1\le j<k\le d,}\ar@{<->}[r]& \sum_{1\le i\le d}\mu_ix_i+\sum_{1\le j<k\le d} \mu_{jk}x_j\wedge x_k.
}
\]
By \eqref{C8:format}, $\GL(d,\pfield)$-action on $\mathbb P_\pfield^{d(d-1)/2}$ is translated into $\PGL(d,\pfield)$-action on $\mathbb P(W)$ by affine automorphisms of $\Lambda(\mathfrak h)$. If we consider those one-dimensional subspaces of $W$ as ``noncommutative'' quadratic curves in the exterior algebra, then $\mathcal H ^2(T_d)/\Aut(T_d)$ are in $1$-$1$ correspondence with isomorphism classes of quadratic curves in $\Lambda(\mathfrak h)$ by affine automorphisms. Then the bijection follows from Theorem \ref{C0:thmB} (c).
\newline

\noindent
\textbf{(a) $\Longleftrightarrow$ (c) of Theorem \ref{C0:thmD}.} Let $G$ be a $p$-group of order $p^{d+1}$ with Frattini group isomorphic to $C_p$. Then, we have the following $p$-group extension:  
\[
\xymatrix{
1\ar[r]
&C_p\ar[r]
&G\ar[r]
&C_p^d\ar[r]
& 1.
}
\]
By taking their group algebras, we get a short exact sequence of finite-dimensional local Hopf algebras:
\[
\xymatrix{
1\ar[r]
&k[C_p]\ar[r]
&k[G]\ar[r]
&k[C_p^d]\ar[r]
& 1.
}
\]

By \cite{Hoch}, we can write the dual Hopf algebra of $k[C_p^d]$ as $u(\mathfrak h)$ for some torus $\mathfrak h$ of dimension $d$. Similarly, we write $(k[C_p])^*$ as $u(\mathfrak g)$ for some one-dimensional torus $\mathfrak g$. Therefore by dualising the above sequence; see \cite[Lemma 4.1]{Byott}, we obtain a short exact sequence in the sense of \eqref{E:SES}. Moreover, the dual Hopf algebra $(k[G])^*$ is a $B$- extension over $A$, whose primitive space is isomorphic to $\mathfrak h$. Hence, it is a semisimple Hopf algebra in $\CH$. Since any semisimple connected Hopf algebra is given by the dual Hopf algebra of $k[G]$ for some $p$-group $G$, the process is reversible.

\begin{ques}
For any non-semisimple type $T$, does $\HG(T)/\Aut(T)$ always classify some geometric object?
\end{ques}

\end{document}